\documentclass[11pt]{amsart}

\def \CC {\mathbb{C}}

\def \GG {\mathbb{G}}

\def \RR {\mathbb{R}}

\def \ZZ {\mathbb{Z}}

\def \Acal {\mathcal{A}}

\def \Ocal {\mathcal{O}}

\def \Vcal {\mathcal{V}}

\def \gfr {\mathfrak{g}}
\def \hfr {\mathfrak{h}}

\def \mfr {\mathfrak{m}}
\def \nfr {\mathfrak{n}}

\def \pfr {\mathfrak{p}}

\def \Gscr {\mathscr{G}}

\def \Mscr {\mathscr{M}}
\def \Nscr {\mathscr{N}}
\def \Oscr {\mathscr{O}}
\def \Pscr {\mathscr{P}}

\def \Rscr {\mathscr{R}}
\def \Sscr {\mathscr{S}}
\def \Tscr {\mathscr{T}}

\def \hbar {\bar{h}}

\def \Gtilde {\widetilde{G}}

\def \Gbf {\mathbf{G}}
\def \Hbf {\mathbf{H}}

\def \Mbf {\mathbf{M}}
\def \Nbf {\mathbf{N}}

\def \Pbf {\mathbf{P}}

\def \Sbf {\mathbf{S}}
\def \Tbf {\mathbf{T}}
\def \Ubf {\mathbf{U}}

\def \Zbf {\mathbf{Z}}

\DeclareMathOperator{\Hrm}{H}

\DeclareMathOperator{\vol}{vol}

\def \riso  {\stackrel{\sim}{\rightarrow}}

\DeclareMathOperator{\im}{im}

\DeclareMathOperator{\coker}{coker}

\DeclareMathOperator{\pr}{pr}

\DeclareMathOperator{\Res}{Res}

\DeclareMathOperator{\id}{id}

\DeclareFontFamily{U}{wncy}{}
\DeclareFontShape{U}{wncy}{m}{n}{<->wncyr10}{}
\DeclareSymbolFont{mcy}{U}{wncy}{m}{n}
\DeclareMathSymbol{\Sha}{\mathord}{mcy}{"58}

\newcommand{\triv}{\mathbbm{1}}

\DeclareMathOperator{\Ad}{Ad}

\newcommand{\gl}{\mathfrak{gl}}

\DeclareMathOperator{\GL}{GL}

\newcommand{\der}{\mathrm{der}}
\newcommand{\ad}{\mathrm{ad}}

\DeclareMathOperator{\Gal}{Gal}

\newcommand{\ur}{\mathrm{ur}}

\DeclareMathOperator{\Tr}{Tr}
\DeclareMathOperator{\Hom}{Hom}

\DeclareMathOperator{\End}{End}

\DeclareMathOperator{\Lie}{Lie}

\newcommand{\rad}{\mathbf{rad}}

\newcommand{\Fil}{\mathrm{Fil}}
\newcommand{\gr}{\mathrm{gr}}

\newcommand{\cl}[1]{\mkern 1.5mu\overline{\mkern-1.5mu#1\mkern-1.5mu}\mkern 1.5mu}
\newcommand{\scl}[1]{{#1}^s} 

\newcommand{\bGL}{\boldsymbol\GL}

\newcommand{\bRes}{\boldsymbol\Res}
\newcommand{\SOcal}{\mathcal{SO}}
\newcommand{\ps}{\mathrm{ps}}

\DeclareMathOperator{\Urm}{U}

\newcommand{\bU}{\boldsymbol\Urm}

\numberwithin{equation}{section}

\theoremstyle{plain}
\newtheorem{thm}[equation]{Theorem}
\newtheorem*{thm*}{Theorem}

\newtheorem{lem}[equation]{Lemma}
\newtheorem*{lem*}{Lemma}

\newtheorem{prop}[equation]{Proposition}

\newtheorem*{prop*}{Proposition}

\newtheorem{cor}[equation]{Corollary}
\newtheorem*{cor*}{Corollary}

\newtheorem*{clm*}{Claim}

\newtheorem*{conj*}{Conjecture}

\theoremstyle{definition}
\newtheorem{defn}[equation]{Definition}
\newtheorem*{defn*}{Definition}

\newtheorem{constr}[equation]{Construction}

\newtheorem*{ass*}{Assumption}

\theoremstyle{remark}
\newtheorem{rmk}[equation]{Remark}
\newtheorem*{rmk*}{Remark}

\newtheorem{exa}[equation]{Example}
\newtheorem*{exa*}{Example}

\numberwithin{figure}{equation}
\numberwithin{table}{equation}
\usepackage{latexsym}
\usepackage{array}

\usepackage[british]{babel}

\usepackage{amsfonts}
\usepackage{amscd}
\usepackage{amsxtra}
\usepackage{amstext}
\usepackage{amssymb}
\usepackage{pifont}
\usepackage{leftidx}
\usepackage{stmaryrd}

\usepackage{mathrsfs}
\usepackage{mathtools}
\usepackage{graphicx}
\usepackage{calc}
\usepackage{url}

\usepackage{csquotes}

\usepackage{pb-diagram,pb-xy}

\usepackage{verbatim}

\usepackage{ifthen}
\usepackage{epsfig}
\usepackage{textcomp}

\usepackage{psfrag}
\usepackage{manfnt}

\usepackage{setspace}
\usepackage{changepage}
\usepackage{layout}

\usepackage[normalem]{ulem}

\usepackage[all,cmtip]{xy}
\xyoption{matrix}

\usepackage{tikz-cd}
\usetikzlibrary{calc}
\usepackage{stackrel}
\usepackage{enumitem}

 \usepackage{xcolor}
 \usepackage{hyperref}
\hypersetup{%
  colorlinks=true,
}

\usepackage[most]{tcolorbox}

\usepackage{xspace}

\usepackage[OT2,T1]{fontenc}

\usepackage{libertine}
\usepackage[libertine]{newtxmath}

\usepackage{bbm}

\usepackage[backend=biber,style=alphabetic,natbib=true, doi=false, isbn=false, url=false]{biblatex}
\hypersetup{%
  colorlinks=true,
}

\begin{document}

\title[Convergence of orbital integrals]{Convergence of orbital integrals on unitary groups in positive characteristic} 
\author[W.~Kim]{Wansu Kim}
\address{Department of Mathematical Sciences, KAIST, 291 Daehak-ro, Yuseong-gu, Daejeon, 34141, Republic of Korea}
\email{wansu.math@kaist.ac.kr}

\author[M.~Park]{Minju~Park}
\address{Department of Mathematical Sciences, KAIST, 291 Daehak-ro, Yuseong-gu, Daejeon, 34141, Republic of Korea}
\email{minjuparkg@gmail.com}

\begin{abstract}
    We prove the absolute convergence of orbital integrals on a unitary group over a non-archimedean local field in any positive characteristic.
\end{abstract}
\maketitle

\tableofcontents

\section{Introduction}

Let $\Gbf$ be a connected reductive group over a non-archimedean local field $F$ of characteristic~$0$, with $G\coloneqq \Gbf(F)$. A fundamental object in the harmonic analysis of $G$ is an \emph{orbital integral}, which is an invariant distribution associated to a conjugacy class in $G$. For semisimple conjugacy classes, the study of orbital integrals was initiated by Harish-Chandra, and the well-definedness (i.e., absolute convergence) of orbital integrals for general conjugacy classes was established by Ranga~Rao \cite{RangaRao:OrbInt}, and independently by Deligne. The proof hinges on establishing the convergence of \emph{unipotent} orbital integrals---since the general case reduces to this via the Jordan decomposition---and utilising the exponential map to translate the problem to the Lie algebra.

However, when $F$ is a non-archimedean local field of characteristic~$p>0$, the aforementioned strategy encounters significant obstacles. While the absolute convergence of \emph{semisimple} orbital integrals follows essentially by the same arguments as in characteristic~$0$, extending this to general conjugacy classes requires substantial modification. In fact, because $F$ is imperfect, the existence of Jordan decomposition over $F$ is not guaranteed, and the exponential map is also problematic in characteristic~$p$.

Nonetheless, McNinch overcame these obstacles and proved the absolute convergence of \emph{unipotent} orbital integrals when the characteristic $p$ of the field avoids a ``small'' list of bad primes determined by the Dynkin diagram of the simple factors of $\Gbf^{\ad}$ and the order of $\pi_1(\Gbf^{\der})$. (See \cite[\S8.6, Theorem~58]{McNinch:NilpOrbits} for the precise statement; for unitary groups, this implies that unipotent orbital integrals absolutely converge in \emph{any} positive characteristic.) The absolute convergence of \emph{general} orbital integrals is also obtained, provided that $p-1$ is strictly greater than the semisimple rank of $\Gbf$. (See \cite[\S8.7, Theorem~61]{McNinch:NilpOrbits}. Note that this assumption on $p$ ensures that any element of $G$ admits a Jordan decomposition over $F$, by \cite[\S8.2, Proposition~48]{McNinch:NilpOrbits}.)

At first glance, one might assume that the absolute convergence of orbital integrals is essentially resolved, given its validity for sufficiently large characteristics. However, when working over a fixed non-archimedean local field $F$ of characteristic $p>0$, McNinch's bounds inherently restrict the applicability of his results to connected reductive groups whose semisimple rank is strictly bounded by $p-1$. Furthermore, this restriction is particularly stark in global applications; unlike number fields, where a condition on $p$ merely excludes finitely many places, all local completions of a global function field share the exact same positive characteristic.

While the aforementioned results rely on restricting the characteristic, a completely characteristic-independent approach has long been known for inner forms of $\GL_n(F)$. Building on ideas of Howe \cite{Howe:FourierTransforms-Germs-GLn}, the absolute convergence of orbital integrals on these groups was established in arbitrary characteristic by Deligne--Kazhdan--Vign\'eras \cite[\S{A.1}]{DeligneKazhdanVigneras:CenSimAlg} (see also Laumon \cite[\S4.8]{Laumon:Book1}). In particular, their approach successfully handles general conjugacy classes without requiring a Jordan decomposition over $F$. The primary goal of this paper is to extend these characteristic-independent methods from inner forms of $\GL_n(F)$ to its \emph{outer forms}---namely, unitary groups.

Here is the main result of the paper.
\begin{thm*}[Theorem~\ref{th:main}]
    Let $G$ be a unitary group over a non-archimedean local field $F$ of any characteristic. Then for any element $\gamma \in G$ and any $f\in C^\infty_c(G)$, the orbital integral $O^G_\gamma(f)$ absolutely converges.\end{thm*}
A natural first attempt is to adapt the proof strategy from \cite[\S{A.1}]{DeligneKazhdanVigneras:CenSimAlg}. Specifically, they observed that absolute convergence is immediate for a broader class of elements beyond just the semisimple ones—namely, \emph{closed} elements, following the terminology of \cite[Definition~4.3.1]{Laumon:Book1}. Their main idea is to associate to each $\gamma\in \GL_n(F)$ a \emph{suitable parabolic subgroup} $P=MN$ such that the absolute convergence of $O^{\GL_n(F)}_\gamma(f)$ is reduced to the absolute convergence of the orbital integral $O^M_{\overline\gamma}(f^P)$, where $\overline\gamma\in M$ is a \emph{closed} element and $f^P$ is a \emph{constant term} of $f$ in the sense of \cite[(4.1.9)]{Laumon:Book1}.

While this approach can be successfully adapted to a unitary group $G$ for \emph{regular} conjugacy classes, it encounters a severe obstruction in general. Unfortunately, the structure of parabolic subgroups is significantly more restrictive for unitary groups than for $\GL_n(F)$, and the scarcity of suitable parabolic subgroups requires a careful restructuring of the proof.

Faced with this scarcity of parabolics, we take inspiration from the nilpotent Lie algebra setting. For any nilpotent element in the Lie algebra of the unitary group $G$, there is a natural construction of an $F$-rational cocharacter---and thus an $F$-rational parabolic subgroup \cite[\S{IV.2.22}]{SpringerSteinberg:ConjClasses}. We adapt this construction and associate a natural $F$-rational parabolic to a significantly wider class of elements in the unitary group, which we introduce as \emph{primary elements} (Definition~\ref{def:primary}).

The central novelty of our approach lies in analysing the conjugacy class of a primary element $\gamma$ with respect to its associated parabolic subgroup $P$. Specifically, we construct a set $\mathcal{V}(\gamma) \subset P$ that is closed in the analytic topology and contains the $P$-conjugacy class of $\gamma$ as an open subset (see Corollary~\ref{cor:V-gamma}). While this construction takes inspiration from McNinch's work on nilpotent Lie algebra orbits \cite[\S8.1, Proposition~46]{McNinch:NilpOrbits}, extending this geometry to elements $\gamma$ that are not merely central translates of unipotent elements requires a delicate treatment. Equipped with a suitable measure on $\mathcal{V}(\gamma)$ (Construction~\ref{constr:measure}), we are then able to adapt the strategy of Ranga~Rao \cite{RangaRao:OrbInt} and McNinch \cite[\S8]{McNinch:NilpOrbits} to deduce absolute convergence.

Although the main result is already known in characteristic~$0$, our proof applies uniformly to this setting as well. Therefore, for conceptual clarity, we allow $F$ to be a non-archimedean local field of arbitrary characteristic throughout the paper.

Furthermore, we expect that the machinery developed here can be extended to other classical groups, provided the characteristic of $F$ is not $2$. The details of this generalisation will appear in subsequent work.

\subsection*{Outline of the paper} After a brief review of unitary groups over a non-archimedean local field in \S\ref{sec:review}, \S\ref{sec:elem-div} establishes the basic properties of conjugacy classes and centralisers within these groups. In \S\ref{sec:main}, we state our main result and reduce the general convergence problem to the case of \emph{primary} elements. The core geometric machinery is developed in the subsequent two sections: in \S\ref{sec:parabolic} we associate a parabolic subgroup to each primary element, and in \S\ref{sec:closure} we execute the main technical step of constructing the subset $\Vcal(\gamma)$. Finally, \S\ref{sec:proof} synthesises these tools to complete the proof of the main theorem.

\subsection*{Notation and preliminaries}
Throughout the paper, let $F$ be a non-archimedean local field, with valuation ring $\Oscr$ and residue field $\kappa$. Let $p$ denote the characteristic of $\kappa$, and $q\coloneqq |\kappa|$. We fix a uniformiser $\varpi\in\Oscr$ once and for all. We normalise the discrete valuation $v\colon F^\times\to\ZZ$ by $v(\varpi)= 1$, and the absolute value by $\lVert \varpi\rVert =  q^{-1}$. If $F$ needs to be specified, we write $\Oscr_F$, $\kappa_F$, $q_F$, $\varpi_F$, $v_F$, $\lVert\bullet\rVert_F$, etc.

We let $\cl F$ (respectively, $\scl F$) denote the algebraic (respectively, separable) closure of $F$. We implicitly embed any finite extension of $F$ into $\cl F$ (respectively, any finite separable extension of $F$ into $\scl F$).

We will use the boldface font to denote linear algebraic groups over $F$ (such as $\Gbf$, $\Hbf$, etc), and use the normal font for the topological groups of $F$-rational points; that is,  $G\coloneqq\Gbf(F)$, $H\coloneqq\Hbf(F)$, etc. We use the gothic font to denote the associated Lie algebra over $F$; that is, $\gfr\coloneqq \Lie(\Gbf)$, $\hfr\coloneqq \Lie(\Hbf)$, etc. For any field extension $F'/F$, we let $\Gbf_{F'}$ and $\gfr_{F'}$ denote the base change of $\Gbf$ and $\gfr$ over $F'$, respectively.

Let $\Hbf$ be a closed $F$-subgroup of $\Gbf$, which is smooth over $F$. On the one hand, the quotient $H\backslash G$ is an \emph{analytic manifold} over $F$ (or an \emph{analytic $F$-manifold}) in the sense of Serre \cite[Part~II, Chap.~III, \S2]{Serre:LieThy}.  We also have the scheme quotient $\Hbf\backslash\Gbf$, which is a smooth affine algebraic variety over $F$, so  $(\Hbf\backslash\Gbf)(F)$ has a natural structure of an analytic $F$-manifold, which contains $H\backslash G$.

\section{Review of unitary groups}\label{sec:review}
Let $E$ be a quadratic separable extension of non-archimedean local fields of residue characteristic $p$. We let $\Oscr'$, $\kappa'$, $q'$, $\varpi'$, $v'$, and $\lVert\bullet\rVert'$ denote the corresponding objects for $E$. We denote the non-trivial element of $\Gal(E/F)$ by $\overline{(\bullet)}$.

We fix a rank-$n$ \emph{non-degenerate hermitian space} $(V,h)$ over $E$; namely, an $n$-dimensional $E$-vector space $V$ together with a non-degenerate hermitian form 
\begin{equation}\label{eq:hermitian-form}
    h\colon V\times V\to E.
\end{equation}
(Our convention is that $h$ is $E$-linear on the \emph{second} argument and $h(w,v) = \overline{h(v,w)}$, following \cite[Ch~7, Definition~1.2]{Scharlau:Forms})
For $g\in \GL_E(V)$, let $g^\star$ denote the adjoint of $g$ with respect to $h$; that is 
\begin{equation}\label{eq:adjoint}
    h(w,gv) = h(g^\star w, v) \quad \forall v,w\in V.
\end{equation}
More explicitly, fix an $E$-basis $(e_1,\cdots,e_n)$ for $V$, and let $J\coloneqq (h(e_i,e_j))_{i,j=1,\cdots,n}$ be the Gram matrix of $h$, so we have $h(w,v) = \overline w^t J v$. Then we have
\begin{equation}\label{eq:star-formula}
    g^\star = J^{-1} \overline g^t  J.
\end{equation}

\begin{defn}
    Given a non-degenerate hermitian form $h$ on $V$, we define the \emph{unitary group} as follows:
    \begin{align*}
        \Urm(V,h)&\coloneqq \{g\in\GL_E(V)\mid h(gv,gw)=h(v,w),\ \forall v,w\in V\}\\
        &= \{g\in\GL_E(V)\mid g^\star g = 1\}.
    \end{align*}

    Then $\Urm(V,h)$ is the group of $F$-rational points of a closed $F$-subgroup $\bU(V,h)$ of $\bRes_{E/F}\bGL_E(V)$.
\end{defn}
For the remainder of the paper, we set $\Gbf \coloneqq \bU(V,h)$ and $G \coloneqq \Gbf(F)$ unless otherwise stated. We will also use the notation $\widetilde\Gbf \coloneqq \bRes_{E/F}\bGL_E(V)$ and $\Gtilde \coloneqq \widetilde\Gbf(F) = \GL_E(V)$ as needed.

Given a non-degenerate hermitian form $h$, we can associate a non-degenerate symmetric $F$-bilinear form
\begin{equation}
 \Psi_h\coloneqq\Tr_{E/F}\circ h \colon V\times V\to F.
\end{equation}
Note that $h$ and $\Psi_h$ induce the same adjoint involution $g \mapsto g^\star$ on $\End_E(V)$. Consequently, the defining condition for $G$ is equivalent to preserving $\Psi_h$; that is, $\Psi_h(gv,gw) = \Psi_h(v,w)$ for all $v,w\in V$. Since $h$ is uniquely determined by $\Psi_h$, it is often convenient to work with the bilinear space $(V,\Psi_h)$ instead of $(V,h)$.

\begin{rmk}
Since $F$ is a non-archimedean local field, any central simple $E$-algebra equipped with an involution restricting to $\overline{(\bullet)}$ on $E$ is \emph{split}; see \cite[Ch.~10, Theorem~(2.2)]{Scharlau:Forms}. 

Moreover, if the characteristic of $F$ is not $2$, we can pass between hermitian and skew-hermitian forms without altering the adjoint involution. Indeed, by choosing a non-zero element $\eta\in E$ with $\overline\eta = -\eta$, the form $h'(v,w)\coloneqq h(v,\eta w)$ defines a non-degenerate skew-hermitian form with the same adjoint involution as $h$. If the characteristic of $F$ is $2$, then any alternating hermitian form $h\colon V\times V\to E$ necessarily satisfies $h(v,w) = \overline{h(w,v)}$.

Consequently, it suffices to restrict our attention to non-degenerate hermitian spaces over $E$.
\end{rmk}

Since we have a natural isomorphism $V\otimes_FE \cong V\times \overline V$, where $\overline V$ is the scalar extension of $V$ by $\overline{(\bullet)}$, we have $\widetilde\Gbf_E\cong \bGL_E(V)\times\bGL_E(\overline V)$. One can show without difficulty that the composition of the following maps
\begin{equation}\label{eq:unitary-BC-to-GL}
    \begin{tikzcd}
        \Gbf_E \arrow[r, hook] & \widetilde\Gbf_E \cong \bGL_E(V)\times\bGL_E(\overline V) \arrow[r, two heads, "\pr"] & \bGL_E(V)
    \end{tikzcd}
\end{equation}
is an isomorphism. In particular, one can naturally identify $\Gbf(E)$ with $\GL_E(V)$.

\section{Centralisers and stable conjugacy classes in unitary groups}\label{sec:elem-div}
Given $\gamma\in\GL_E(V)$, one can naturally define an $E[T]$-module structure on $V$ by letting $T$ act via $\gamma$. This is a standard and powerful tool for analysing the structure of conjugacy classes and centralisers of $\gamma$ in $\GL_E(V)$. In this section, we refine this approach to extract the structural properties of conjugacy classes and centralisers of a unitary group $G=\Urm(V,h)$. 

We continue with the notation from \S\ref{sec:review}, where $\Gbf = \bU(V,h)$ and $G=\Gbf(F)$.

\begin{defn}
    For $\gamma\in G$, let $\Gbf_\gamma \subset \Gbf$ denote the centraliser of $\gamma$ as a group scheme over $F$, and write $G_\gamma\coloneqq\Gbf_\gamma(F)$. Note that $G_\gamma \subset G$ coincides with the centraliser of $\gamma\in G$ as a topological group.
\end{defn}

To study $\Gbf_\gamma$ for $\gamma\in G$, we first analyse its base change $\Gbf_{\gamma,E}$ to $E$. In fact, centralisers in $\Gbf_E\cong \GL_E(V)$ are well understood (see \cite[Ch~4]{Laumon:Book1}), and the formation of centralisers commutes with base change (i.e., $\Gbf_{\gamma,E}\cong \Gbf_{E,\gamma}$).

Fix an element $\gamma\in \Gbf(E) \cong \GL_E(V)$, and view $V$ as an $E[T]$-module where $T$ acts via $\gamma$. Then, by the structure theorem for finitely generated $E[T]$-modules, one can write
\begin{equation}
    V \cong \bigoplus_\wp V\{\wp\}
\end{equation}
where $\wp\in E[T]$ are monic irreducible polynomials and $V\{\wp\}\cong \bigoplus_i E[T]/\wp^{m_i}$ is the $\wp$-primary part of $V$. Furthermore, for any $E$-algebra $R$, we have a natural isomorphism 
\begin{equation}\label{eq:GL-centraliser}
    \Gbf_{E,\gamma}(R) \cong \End_{R[T]}(V\otimes_ER)^\times \cong  (C_\gamma\otimes_ER)^\times,
\end{equation}
where $C_\gamma\coloneqq \End_{E[T]}(V)$ is the \emph{commutant} of $\gamma$ in $\End_E(V)$. 

Let us now describe the $E$-rational unipotent radical $\Rscr_{u,E}\Gbf_{E,\gamma}$ and the maximal pseudo-reductive quotient $\Gbf_{E,\gamma}^{\ps}\coloneqq \Gbf_{E,\gamma}/\Rscr_{u,E}\Gbf_{E,\gamma}$. 

For an irreducible polynomial $\wp\in E[T]$, we also let $E_\wp\coloneqq E[T]/\wp$, which is a finite extension of $E$. Then the maximal semisimple quotient $C_\gamma/\rad(C_\gamma)$ of $C_\gamma$ can be written as follows.
\begin{eqnarray}
    C_\gamma/\rad(C_\gamma) = \prod_\wp\End_{E_\wp}(V\{\wp\}/\wp V\{\wp\})
\end{eqnarray}
It then follows that $\Gbf_{E,\gamma}^{\ps}$ is the unit group of $C_\gamma/\rad(C_\gamma)$, which is isomorphic to
\begin{equation}\label{eq:ps-red-qt-centraliser-GL}
    \Gbf_{E,\gamma}^\ps \cong \prod_\wp \bRes_{E_\wp/E}\bGL_{E_\wp}(V\{\wp\}/\wp V\{\wp\}).
\end{equation}

Next, we have 
\begin{equation}\label{eq:rat-unip-rad-centraliser-GL}
    \Rscr_{u,E}\Gbf_{E,\gamma}(\cl F) \cong \ker \left(\Gbf_{E,\gamma}(\cl F)\to \Gbf_{E,\gamma}^{\ps}(\cl F)\right) = 1+\rad(C_\gamma)\otimes_E\cl F;
\end{equation}
that is, $\Rscr_{u,E}\Gbf_{E,\gamma}$ corresponds to the multiplicative group $1+\rad(C_\gamma)$. Furthermore, the filtration $\left( 1+\rad(C_\gamma)^i \right)_{i>0}$ of $1+\rad(C_\gamma)$ induces a filtration of $\Rscr_{u,E}\Gbf_{E,\gamma}$, whose  successive quotients are vector groups over $E$. This shows that $\Rscr_{u,E}\Gbf_{E,\gamma}$ is a smooth connected $E$-split unipotent group.  

Let us now describe the $F$-rational structure of $\Gbf_\gamma$ for $\gamma\in G=\Urm(V,h)$. Firstly, $h\colon V\times V \to E$ induces the following $E$-linear isomorphism
\begin{equation}\label{eq:h-hat}
    \hat h\colon
    \begin{tikzcd}
        V \arrow[r,"\cong"] &\overline V^\vee
    \end{tikzcd} \quad \text{with }\hat h(w)\coloneqq h(w,-)\colon V \to E.
    \end{equation} 
    where $\overline V^\vee$ is the scalar extension of $V^\vee$ by $\overline{(\bullet)}\in\Gal(E/F)$. (Recall that $h$ is $E$-linear on the \emph{second} argument.) We endow $\overline V^\vee$ with the contragradient action of $\Gbf(E)=\GL_E(V)$; that is, for any $g\in \GL_E(V)$ and $l\in \overline V^\vee$ we set 
    \begin{equation}\label{eq:contragradient}
        g\cdot l \colon 
        \begin{tikzcd}
            V \arrow[r,"g^{-1}"] & V \arrow[r,"l"] & E
        \end{tikzcd}.
    \end{equation}
    Then we have 
    \[
    \hat h (gw) = (g^\star)^{-1}\hat h(w), \quad \forall g\in \Gbf(E), \forall w\in V,\]
    so \emph{$h$ commutes with $\gamma\in \GL_E(V)$ if and only if $\gamma^\star \gamma=1$}; ie, $\gamma\in G$. Thus, if we endow $V$ and $\overline V^\vee$ with an $E[T]$-module structure by letting $T$ act via $\gamma\in G$, then $\hat h\colon V\to \overline V^\vee$ is an $E[T]$-module isomorphism.
    
    To describe $\overline V^\vee\{\wp\} = \hat h (V\{\wp\})$, we introduce the following:
\begin{defn}
    For a monic polynomial $f(T) \in E[T]$ of degree $d$ with $f(0) \ne 0$, we define the \emph{dual polynomial} as
    \[
    f^\vee(T) \coloneqq f(0)^{-1} T^d \cdot f(T^{-1}).
    \]
    We also define  $\overline f(T)$ by applying $\overline{(\bullet)}$ to each coefficient of $f(T)$. 
\end{defn}
For two monic polynomials $f_1,f_2\in E[T]$ with non-zero constant terms, we have $(f_1f_2)^\vee = f_1^\vee\cdot f_2^\vee$ and $\overline{f_1f_2} = \overline{f_1}\cdot\overline{f_2}$. In particular, if $\wp$ is a monic irreducible polynomial, then so are $\overline\wp$, $\wp^\vee$, and $\overline\wp^\vee$. If $f(T)$ is the characteristic polynomial of $\gamma\in\GL_E(V)$, then  $\overline f^\vee(T)$ is the characteristic polynomial of $(\gamma^{-1})^\star$. 

\begin{prop}\label{prop:elem-div-self-dual}
    Fix $\gamma\in G$, and endow $V$ and $\overline V^\vee$ with an $E[T]$-module structure via $\gamma$. Let $\wp\in E[T]$ be a monic irreducible polynomial. Then, we have 
    \[
    \overline V^\vee \{\overline\wp^\vee\} = \overline{V\{\wp\}}^\vee;
    \]
    in other words, a semilinear functional $l\in \overline V^\vee$ is $\overline\wp^\vee$-primary if and only if we have $l( V\{\wp'\}) = 0 $ for any monic irreducible polynomial $\wp'\ne \wp$.
\end{prop}
\begin{proof}
    Recall that the action of $\gamma\in G$ on $l\in \overline V^\vee$ is given as follows
    \[\forall v\in V,\quad (\gamma\cdot l)(v)\coloneqq  l(\gamma^{-1}v) = l(\gamma^\star v)\qquad \text{(see \eqref{eq:contragradient}).}\]
    Thus, for any $f(T)\in E[T]$, we have 
    \[\forall v\in V,\quad \left( f(\gamma)\cdot l \right)(v) = l\big( \overline f(\gamma^{-1})v \big) = l\big( \overline f(\gamma^\star)v \big).\]

    Let $d_\wp\coloneqq\deg\wp$. Observe that $\overline \wp(\gamma^{-1}) = \overline\wp(0)\gamma^{-d_\wp}\overline\wp^\vee(\gamma)$, and $\overline\wp(0)\gamma^{-d_\wp}$ is invertible. Therefore, it follows that $l\in\overline V^\vee$ is annihilated by $\overline\wp^\vee(\gamma)^m$ for some $m$ if and only if $l(\wp(\gamma)^m v)=0$ for any $v\in V$. Now, since $\wp(\gamma)$ acts invertibly on $V\{\wp'\}$ for any monic irreducible $\wp' \ne \wp$, the proposition follows.
\end{proof}

\begin{subequations}\label{eq:elem-div}
In the same setting as Proposition~\ref{prop:elem-div-self-dual}, choose a finite index set $\Sigma$ and a monic irreducible polynomial $\wp_a\in E[T]$ for each $a\in \Sigma$ such that $V\{\wp_a\}\ne 0$ and we have
\begin{equation}\label{eq:elem-div:Sigma}
    V = \bigoplus_{a\in\Sigma}V_a,\quad\text{where }V_a\coloneqq V\{\wp_a\}+V\{\overline\wp_a^\vee\}.
\end{equation} 
Let $\Sigma_{nd}$ be the subset of $\Sigma$ consisting of indices $a$ such that $\wp_a=\overline\wp_a^\vee$, and set $\Sigma_{h}\coloneqq \Sigma\setminus\Sigma_{nd}$. Then, by definition we have
\begin{equation}
    V_a = \begin{cases}
        V\{\wp_a\} & \text{if }a\in\Sigma_{nd};\\
        V\{\wp_a\}\oplus V\{\overline\wp_a^\vee\} & \text{if }a\in\Sigma_h.
    \end{cases}
\end{equation}
For each $a\in\Sigma$, we write
\begin{equation}\label{eq:elem-div:Va}
    V\{\wp_a\} \cong \bigoplus_{i} \Big( E[T]/\wp_a^{m_{a,i}} \Big)^{r_{a,i}} \quad \text{and}\quad V\{\overline\wp_a^\vee\} \cong \bigoplus_{i} \Big( E[T]/{\overline\wp_a^{\vee }}^{m_{\overline a,i}} \Big)^{r_{\overline a,i}}.
\end{equation}
Here, we order the exponents as $m_{a,1}>m_{a,2}>\cdots>0$, and  $r_{a,i}$ denotes the multiplicity. We apply the same convention to $m_{\overline {a},i}$'s and $r_{\overline a,i}$'s if $a\in\Sigma_h$ and $V\{\overline\wp_a^\vee\}\ne0$. Using this notation, we record some consequences of Proposition~\ref{prop:elem-div-self-dual}.
\end{subequations}

\begin{cor}\label{cor:elem-div-self-dual}
    Using the above notation, the following properties are valid.
    \begin{enumerate}
        \item\label{cor:elem-div-self-dual:isotropic} For any $a\in \Sigma_h$, we have $V\{\overline\wp_a^\vee\}\ne0$, with $m_{a,i} = m_{\overline a,i}$ and $r_{a,i} = r_{\overline a,i}$ for each $i$. Furthermore, both $V\{\wp_a\}$ and $V\{\overline\wp_a^\vee\}$ are totally isotropic, and $h$ restricts to a non-degenerate hermitian form $h_a$ on $V_a$.
        \item\label{cor:elem-div-self-dual:nondeg} For any $a\in\Sigma_{nd}$, $h$ restricts to a non-degenerate hermitian form $h_a$ on~$V_a = V\{\wp_a\}$. 
        \item\label{cor:elem-div-self-dual:ortho-decomp}
        We have $(V,h) = \bigoplus_{a\in\Sigma}(V_a,h_a)$, where the direct sum is orthogonal.
        Furthermore, the  involution $\star$ on $C_\gamma$ sends $\End_{E[T]}\big(V\{\wp_a\}\big)$ onto $\End_{E[T]}\big(V\{\overline\wp_a^\vee\}\big)$ and vice versa. 
        \item\label{cor:elem-div-self-dual:centraliser-ps} $\Gbf_\gamma^{\ps}$ is a product of Weil restrictions of general linear groups and unitary groups. More precisely, its group of $F$-points is
        \[G_\gamma^\ps \cong \prod_{a\in\Sigma_h} \GL_{E_a}\left( V\{\wp_a\}/\wp_a V\{\wp_a\} \right)\times \prod_{a\in\Sigma_{nd}}\Urm_{E_a}\left( V_a/\wp_a V_a\right),\]
        where $E_a \coloneq E[T]/\wp_a$ and $\Urm_{E_a}\left( V_a/\wp_a V_a \right) \coloneq \left\{g\in \GL_{E_\wp}(V\{\wp\}/\wp V\{\wp\})\mid g^\star g = 1\right\}$. Here, $\star$ is the involution on $C_\gamma/\rad(C_\gamma)$ induced by $\star$ on $C_\gamma$.
    \end{enumerate}
\end{cor}
\begin{proof}
    Since $\hat h\colon V\riso \overline V^\vee$ is an $E[T]$-module isomorphism, it restricts to an isomorphism on $\wp_a$-primary parts for each $a\in \Sigma$. Since $\overline V^\vee\{\wp_a\}$ is the conjugate-dual of $V\{\overline\wp_a^\vee\}$ by Proposition~\ref{prop:elem-div-self-dual}, we have $V\{\overline\wp_a^\vee\}\ne 0$ and the equalities $m_{a,i}=m_{\overline a,i}$ and $r_{a,i} = r_{\overline a,i}$ follow for each $i$.
    
    If $a\in\Sigma_h$ so $\wp_a\ne\overline\wp_a^\vee$, then the previous paragraph implies that $V\{\wp_a\}$ and $V\{\overline\wp_a^\vee\}$ are hyperbolic pair of totally isotropic subspaces of $V$. This proves (\ref{cor:elem-div-self-dual:isotropic}).

    If $a\in\Sigma_{nd}$ so $\wp_a = \overline\wp_a^\vee$, then $\hat h$ restricts to $V\{\wp_a\}\riso \overline{V\{\wp_a\}}^\vee$, so $h$ restricts to a non-degenerate hermitian form $h_a$ on $V\{\wp\}$. This proves (\ref{cor:elem-div-self-dual:nondeg}), and claim (\ref{cor:elem-div-self-dual:ortho-decomp}) is now straightforward.

    To prove (\ref{cor:elem-div-self-dual:centraliser-ps}), note that 
    \[G_\gamma = \Gbf_\gamma(E)\cap G = C_\gamma^\times\cap G = \{g\in C_\gamma^\times\mid g^\star g=1\}.\]
    And since $\rad(C_\gamma)$ is stable under $\star$, the adjoint involution $\star$ defines an involution on $C_\gamma/\rad(C_\gamma)$, and we have \[G_\gamma^\ps = \{g\in (C_\gamma/\rad(C_\gamma))^\times\mid g^\star g=1\}.\]
    On the direct factor $\GL_{E_a}(V\{\wp_a\}/\wp_a V\{\wp_a\})\times \GL_{E_{\overline a}}(V\{\overline\wp_a^\vee\}/\overline\wp_a^\vee V\{\overline\wp_a^\vee\})$ for $a\in \Sigma_h$ where $E_{\overline a}\coloneqq E[T]/\overline\wp_a^\vee$, the condition $g^\star g= 1$ defines a subgroup that projects isomorphically onto the first factor by the same argument as \eqref{eq:unitary-BC-to-GL}. On the direct factor $\GL_{E_a}(V_a/\wp_a V_a)$ for $a\in \Sigma_{nd}$, the condition $g^\star g = 1$ defines a unitary subgroup as $\star$ does not fix $E_a$. This proves (\ref{cor:elem-div-self-dual:centraliser-ps}).
\end{proof}

We are now ready to obtain the following proposition.

\begin{prop}\label{prop:unitary-centralisers}
    For any $\gamma\in G$, the following properties hold for $\Gbf_\gamma$.
    \begin{enumerate}
        \item\label{prop:unitary-centralisers:sm-conn} $\Gbf_\gamma$ is smooth and geometrically connected over $F$.
        \item\label{prop:unitary-centralisers:unip-rad} The $F$-rational unipotent radical $\Rscr_{u,F}\Gbf_\gamma$ is smooth, connected, and $F$-split.
        \item\label{prop:unitary-centralisers:Gal-coho} $\Hrm^1(F,\Gbf_\gamma)$ is finite, and $\Hrm^1(E,\Gbf_\gamma)$ is trivial. In particular, the inflation map $\Hrm^1(E/F,\Gbf_\gamma(E))\to\Hrm^1(F,\Gbf_\gamma)$ is bijective.
        \item\label{prop:unitary-centralisers:unimodular} The topological group $G_\gamma$ is unimodular.
    \end{enumerate}
\end{prop}
\begin{proof}
     Since $E/F$ is separable, we have $(\Rscr_{u,F}\Gbf_\gamma)_E\cong \Rscr_{u,E}\Gbf_{E,\gamma}$, which is shown to be a smooth connected $F$-split unipotent group using the description \eqref{eq:rat-unip-rad-centraliser-GL}. This implies (\ref{prop:unitary-centralisers:unip-rad}). (Note that the splitness of unipotent groups is insensitive of finite separable base change by \cite[Theorem~B.2.5]{ConradGabberPrasad:PRedGp2ed}.) 

    By Corollary~\ref{cor:elem-div-self-dual}(\ref{cor:elem-div-self-dual:centraliser-ps}), $\Gbf_\gamma^\ps$ is a product of Weil restrictions of connected reductive groups, so $\Gbf_\gamma^\ps$ is \emph{connected} by \cite[Proposition~A.5.9]{ConradGabberPrasad:PRedGp2ed}. Now (\ref{prop:unitary-centralisers:sm-conn}) follows since both $\Gbf_\gamma^\ps$ and $\Rscr_{u,F}\Gbf_\gamma$ are smooth and geometrically connected.

    By the $F$-splitness of $\Rscr_{u,F}\Gbf_\gamma$ and a simple d\'evissage argument, we get
    \begin{equation}
        \begin{tikzcd}
            \Hrm^1(F,\Gbf_\gamma) \arrow[r,"\cong"] & \Hrm^1(F,\Gbf_\gamma^\ps),
        \end{tikzcd}
    \end{equation}
    and the target of the isomorphism is a finite pointed set by the non-abelian Shapiro lemma for any finite Weil restriction of scalars \cite[Lemma~4.1.6]{Conrad:Finiteness-Sha} and the finiteness of the first Galois cohomology of any connected reductive group. Furthermore, the same argument together with \eqref{eq:ps-red-qt-centraliser-GL} and non-abelian Hilbert~Satz~90 shows that $\Hrm^1(E,\Gbf_\gamma)$ is trivial, so the inflation map $\Hrm^1(E/F,\Gbf_\gamma(E))\to\Hrm^1(F,\Gbf_\gamma)$ is bijective. This proves (\ref{prop:unitary-centralisers:Gal-coho}).

    To prove (\ref{prop:unitary-centralisers:unimodular}), note that the modulus character $\delta_{G_\gamma}$ can be computed as follows:
    \[
        \delta_{G_\gamma}(g) = \lVert \det\big(\Ad _{\gfr_\gamma}(g)\big) \rVert ^{-1}, \quad \forall g\in G_\gamma,
    \]
    where $\gfr_\gamma\coloneqq\Lie(\Gbf_\gamma)$, and $\Ad_{\gfr_\gamma}(g)$ denotes the adjoint action of $g$ on $\gfr_\gamma$ (see \cite[Chap.~III, \S3, Cor of Prop~55]{Bourbaki:Lie2-3} for details). Therefore, to show that $G_\gamma$ is unimodular, it suffices to show that $\Gbf_\gamma(E)$ is unimodular, which is proved in \cite[Lemma~(4.8.6)]{Laumon:Book1}. This proves (\ref{prop:unitary-centralisers:unimodular}).
\end{proof}

\begin{defn}
    For $\gamma\in G$, let $\Ocal_G(\gamma)$ denote the conjugacy class of $\gamma$ in $G$. Let $c_\gamma\colon\Gbf\to\Gbf$ denote the conjugation map $g\mapsto g^{-1}\gamma g$. This induces a locally closed immersion of analytic manifolds
    \[
        c_\gamma \colon 
        \begin{tikzcd}
            G_\gamma\backslash G \arrow[r,hook] & G,
        \end{tikzcd}
    \]
    which is a homeomorphism onto its image $\Ocal_G(\gamma)$.

    Similarly, viewing $\gamma$ as an element of $\Gbf(E)$, we can consider its conjugacy class $\Ocal_{\Gbf(E)}(\gamma)$. We define the \emph{stable conjugacy class} of $\gamma$ in $G$ by
    \[
        \SOcal_G(\gamma)\coloneqq \Ocal_{\Gbf(E)}(\gamma)\cap G.
    \]
\end{defn}
Since $\Gbf_\gamma(E) \cong \GL_E(V)$, where the notions of conjugacy and stable conjugacy coincide, our definition of stable conjugacy class agrees with the standard one when $F$ has characteristic~$0$ (see \cite[\S3, pp.~788--789]{Kottwitz:RatlConjRedGp}). 
Furthermore, we will see in Proposition~\ref{prop:stable-conjugacy} that $\SOcal_G(\gamma)$ satisfies the analogous property to a stable conjugacy class in characteristic~$0$.

Before stating the proposition, let us recall the following standard fact on conjugacy and stable conjugacy classes of $\gamma\in\Gbf(E) \cong \GL_E(V)$.
\begin{lem}\label{lem:ConjCl-GL}
    For any field extension $E'/E$, an element $\gamma'\in\Gbf(E')$ belongs to $\Ocal_{\Gbf(E')}(\gamma)$ if and only if the $E'[T]$-module structures on $V\otimes_E E'$ defined by $\gamma$ and $\gamma'$ are isomorphic. In particular, if $E'/E$ is a Galois extension, then we have 
    \[
    \Ocal_{\Gbf(E)}(\gamma) = \Ocal_{\Gbf(E')}(\gamma)^{\Gal(E'/E)} = \Ocal_{\Gbf(E')}(\gamma) \cap \Gbf(E)\]
    as a subset of $\Gbf(E')$.
\end{lem}
\begin{proof}
    Let $M$ (respectively,  $M'$) denote the $E'[T]$-modules with underlying $E'$-vector space $V\otimes_EE'$ induced by $\gamma$ (respectively, $\gamma'$). Then $g\in \Gbf(E') \cong \GL_{E'}(V\otimes_EE')$ induces an $E'[T]$-isomorphism $M\riso M'$ if and only if we have $\gamma'g = g\gamma$, which proves the first claim. This also implies $\Ocal_{\Gbf(E)}(\gamma) \supseteq \Ocal_{\Gbf(E')}(\gamma) \cap \Gbf(E)$, and the reverse inclusion is trivial.
\end{proof}

\begin{subequations}
    \addtocounter{equation}{-1}
\begin{prop}\label{prop:stable-conjugacy}
The stable conjugacy class $\SOcal_G(\gamma)$ is the set of $F$-points of a locally closed subvariety of $\Gbf$ which is smooth over $F$, so $\SOcal_G(\gamma)$ is a locally closed analytic $F$-submanifold of $G$. Furthermore, there is a continuous surjection
\begin{equation}\label{eq:stable-conjugacy}
    \begin{tikzcd}
    \SOcal_G(\gamma) \arrow[r, two heads] & \ker\left(\Hrm^1\big(E/F, \Gbf_\gamma(E)\big) \to \Hrm^1\big(E/F,\Gbf(E)\big) \right)
\end{tikzcd}
\end{equation}
where the target is endowed with the discrete topology, such that each fibre is a single conjugacy class in $G$. In particular, $\SOcal_G(\gamma)$ is a finite disjoint union of conjugacy classes in $G$, where each conjugacy class is an open and closed $F$-submanifold.
\end{prop}
\begin{proof}
    Since $\Gbf_\gamma$ is smooth, the conjugation map 
    $c_\gamma\colon \Gbf_\gamma\backslash\Gbf\to\Gbf$ is an unramified monomorphism of smooth affine $F$-varieties, so it is a locally closed immersion. Thus, to show the first claim on $\SOcal_G(\gamma)$ it suffices to show that $c_\gamma$ induces a homeomorphism
    \begin{equation}\label{eq:stable-conjugacy:scheme-qt}
        c_\gamma\colon\begin{tikzcd}
       (\Gbf_\gamma\backslash\Gbf)(F) \arrow[r,"\cong"] & \SOcal_G(\gamma) 
    \end{tikzcd}.
    \end{equation}
    In fact, by smoothness of $\Gbf_\gamma$ we have a natural bijection
    \[
    \begin{tikzcd}
    \Gbf_\gamma(\scl F)\backslash\Gbf(\scl F) \arrow[r]& (\Gbf_\gamma\backslash\Gbf)(\scl F),
    \end{tikzcd}
    \]
    which in turn restricts to the following bijection by the triviality of $\Hrm^1(E,\Gbf_\gamma)$ (Proposition~\ref{prop:unitary-centralisers}(\ref{prop:unitary-centralisers:Gal-coho})): 
    \[\begin{tikzcd}
        \Gbf_\gamma(E)\backslash\Gbf(E) \arrow[r,"\cong"] &(\Gbf_\gamma\backslash\Gbf)(E)
    \end{tikzcd}.\]
    Now, we have 
    \begin{align*}
        \SOcal_G(\gamma) &= \Ocal_{\Gbf(E)}(\gamma) \cap G = c_\gamma\left( (\Gbf_\gamma\backslash\Gbf)(E) \right) \cap G 
        \\&= c_\gamma\left( (\Gbf_\gamma\backslash\Gbf)(E)\right)^{\Gal(E/F)} = c_\gamma\left( (\Gbf_\gamma\backslash\Gbf)(F) \right).
    \end{align*}
    In other words, $c_\gamma$ induces the homeomorphism \eqref{eq:stable-conjugacy:scheme-qt}, as desired.

    The long exact sequence of pointed sets for the non-abelian Galois cohomology \cite[I.~\S5.4, Proposition~36]{Serre:GalCohom} yields
    \begin{equation}\label{eq:stable-conjugacy:les}
        1\to G_\gamma \backslash G \to (\Gbf_\gamma\backslash\Gbf)(F) \xrightarrow{\delta} \Hrm^1\left(E/F,\Gbf_\gamma(E)\right) \to \Hrm^1\left(E/F,\Gbf(E)\right).
    \end{equation}
    The boundary map $\delta$ can be described as follows. If $g\in \Gbf(E)$ such that $\Gbf_\gamma(E) g$ is $\Gal(E/F)$-stable, then we define the following $1$-cocycle
    \[
    \Gal(E/F)\to\Gbf_\gamma(E);\qquad \overline{(\bullet)}\mapsto g^{-1}\cdot\overline g,
    \]
    which is independent of the choice of a coset representative up to $1$-coboundary. This construction defines a map $\delta$. Furthermore, one can directly show that 
    \[\delta(\Gbf_\gamma(E)g) = \delta(\Gbf_\gamma(E)g') \text{ if and only if } g' = gh \text{ for some }h\in G\]
    (see \cite[\S3]{Kottwitz:RatlConjRedGp}).

    Now, the natural topology on $\Hrm^1(E/F,\Gbf_\gamma(E))$ induced from the analytic topology on $\Gbf_\gamma(E)$ is Hausdorff (the subgroup of $1$-coboundaries is closed in the group of $1$-cocycles), so it is discrete by finiteness. As $\delta$ is continuous by construction, each fibre of $\delta$ is open and closed. Therefore, by precomposing $c_\gamma^{-1}$ from \eqref{eq:stable-conjugacy:scheme-qt} with $\delta$, we obtain the continuous surjection \eqref{eq:stable-conjugacy}. Lastly, any fibre of $\delta$ is an $G$-orbit in $(\Gbf_\gamma\backslash\Gbf)(F)$ for the right translation, which corresponds to a $G$-conjugacy class in $\SOcal_G(\gamma)$ via $c_\gamma$. This concludes the proof.
\end{proof}
\end{subequations}

The following terminology was introduced in \cite[Definition~4.3.1]{Laumon:Book1} for $\gamma\in\GL_E(V)$, which we extend to unitary groups.
\begin{defn}\label{def:closed}
    We say that $\gamma\in G$ is \emph{closed} if $\Ocal_G(\gamma)$ is a closed subset of $G$ for the analytic topology; or equivalently, if $c_\gamma\colon G_\gamma\backslash G\to G$ is a proper morphism (i.e., the preimage of a compact subset is compact). 
    
    Similarly, we say that $\gamma\in \Gbf(E)$ is \emph{closed} if $\Ocal_{\Gbf(E)}(\gamma)$ is a closed subset of $\Gbf(E)$ for the analytic topology.
\end{defn}

\begin{lem}\label{lem:closed}
    \begin{enumerate}
        \item\label{lem:closed:GL} An element $\gamma\in\Gbf(E)\cong \GL_E(V)$ is closed if and only if the minimal polynomial of $\gamma$ is a product of distinct monic irreducible polynomials in $E[T]$; or equivalently, $\Gbf_{E,\gamma}$ is pseudo-reductive.
        \item\label{lem:closed:unitary} An element $\gamma \in G$ is closed if it is closed as an element of $\Gbf(E)$.
        \item\label{lem:closed:ss} If $\gamma\in G$ is semisimple, then $\gamma$ is closed both as an element of $G$ and $\Gbf(E)$.
    \end{enumerate}
\end{lem}
\begin{proof}
    Claim~(\ref{lem:closed:GL}) follows from \cite[Lemma~(4.3.2)(i)]{Laumon:Book1}. (Note that $\Gbf_{E,\gamma}$ is pseudo-reductive if and only if the commutant $C_\gamma\subset \End_E(V)$ of $\gamma$ is semisimple, which is equivalent to the minimal polynomial of $\gamma$ being a product of distinct irreducible polynomials -- see \eqref{eq:ps-red-qt-centraliser-GL} and the discussions above it.)

    Let $\gamma\in G$ be an element such that $\Ocal_{\Gbf(E)}(\gamma)$ is closed in $\Gbf(E)$. Then $\SOcal_G(\gamma) = \Ocal_{\Gbf(E)}(\gamma) \cap G$ is closed in $G$ since $G$ is a closed subgroup of $\Gbf(E)$. Since $\Ocal_G(\gamma)\subset\SOcal_G(\gamma)$ is open and closed by Proposition~\ref{prop:stable-conjugacy}, it follows that $\Ocal_G(\gamma)$ is closed in $G$, which proves (\ref{lem:closed:unitary}).

    Lastly, $\gamma\in \Gbf(E)\cong \GL_E(V)$ is semisimple if and only if its minimal polynomial is separable. So a semisimple element of $G$ is a closed element of  $\Gbf(E)$ by (\ref{lem:closed:GL}), so it is a closed element of $G$ by (\ref{lem:closed:unitary}). This proves (\ref{lem:closed:ss}).
\end{proof}

\begin{rmk}
    It turns out that the converse of Lemma~\ref{lem:closed}(\ref{lem:closed:unitary}) also holds, but we do not need this property.
\end{rmk}

\section{Main results}\label{sec:main}
We are now in a position to formally define orbital integrals in our setting and state the main result of this paper. After observing that absolute convergence is immediate for \emph{closed} elements, we state the general absolute convergence theorem for arbitrary elements in unitary groups (see Theorem~\ref{th:main}). We then perform a preliminary reduction step, showing that the general convergence problem can be reduced to the case of \emph{primary} elements in the sense of Definition~\ref{def:primary}.

Let us first recall the definition of orbital integrals. For convenience in reduction steps, we define orbital integrals on  more general groups than unitary groups $\Urm(V,h)$ and general linear groups $\GL_E(V)$. 

Let $\Gbf$ be a connected reductive group over $F$ such that $\Gbf^{\der}$ is \emph{simply connected} and that the centraliser $G_{\gamma}$ of any $\gamma\in G= \Gbf(F)$ is \emph{unimodular}. 

Suppose that we have
\begin{equation}\label{eq:prod-group}
    \Gbf\coloneqq \prod_{a\in \Sigma} \Gbf_a.\quad \text{and}\quad G\coloneqq \prod_{a\in \Sigma} G_a.
\end{equation}
where $\Sigma$ is a finite set, and $\Gbf_a$ for $a\in \Sigma$ is either a general linear group or a unitary group over $F$. (This is the setting relevant to the later discussion in this paper.) Note that any $\gamma\in G$ can be written as $\gamma = (\gamma_a)_{a\in\Sigma}$, and the centraliser $G_\gamma$ is the product of $(G_a)_{\gamma_a}$, which is unimodular by Proposition~\ref{prop:unitary-centralisers}(\ref{prop:unitary-centralisers:unimodular}). Furthermore, the the space $C^\infty_c(G)$ of $\CC$-valued locally constant compactly supported functions can be written as follows:
\[C^\infty_c(G) \cong \bigotimes_{a\in\Sigma}C^\infty_c (G_a).\]

\begin{defn}
    Let $\gamma\in G$. With respect to the fixed choice of the Haar measures $dg$ of $G$ and $dg_\gamma$ of $G_\gamma$, we define the \emph{orbital integral} of $f\in C^\infty_c(G)$ as follows
    \[
    O^G_\gamma(f)\coloneqq \int_{G_\gamma\backslash G}f(g^{-1}\gamma g) \frac{dg}{dg_\gamma},
    \]
    provided that the integral absolutely converges.
\end{defn}

Since each $\Gbf^{\der}$ is simply connected, our definition of an orbital integral is compatible with the standard definition when $F$ is of characteristic~$0$. If $G$ is of the form as in \eqref{eq:prod-group} and $f = \bigotimes_{a\in\Sigma}f_a$ for $f_a\in C^\infty_c(G_a)$, then we have
\[O^G_\gamma(f) = \prod_{a\in\Sigma}O^{G_a}_{\gamma_a}(f_a),\]
provided that either side of the equation converges.

Clearly, $O^G_\gamma(f)$ depends on the choice of the measure $\frac{dg}{dg_\gamma}$, which is suppressed from the notation. Furthermore, we have  $O^G_\gamma(f) = O^G_{\gamma'}(f)$ for any $\gamma'\in\Ocal_G(\gamma)$; that is,  $O^G_\gamma(f)$ depends only on $\Ocal_G(\gamma)$. 

Recall that any reductive $F$-group is unimodular, so by unimodularity of $\Gbf_\gamma$, both Haar measures $dg$ and $dg_\gamma$ are bi-invariant. Thus, if the orbital integral $O^G_\gamma(f)$ absolutely converges for each $f\in C^\infty_c(G)$, then the functional $O^G_\gamma\colon C^\infty_c(G)\to\CC$ is an \emph{invariant distribution}; that is,
\begin{equation}\label{eq:inv-distrib}
    O^G_\gamma(f) = O^G_\gamma(\prescript{h}{}{f}) \ \text{for any } h\in G,\ \text{where }\prescript{h}{}{f}(g) \coloneqq f(h^{-1}gh).
\end{equation}

\begin{lem}\label{lem:closed-orb-int}
    Let  $\gamma\in G$ be \emph{closed}. Then, for any $f\in C^\infty_c(G)$, the orbital integral $O^G_\gamma(f)$ converges absolutely. In fact, it reduces to a finite sum.
\end{lem}
\begin{proof}
    By closedness, the conjugation map $c_\gamma\colon G_\gamma\backslash G \to G$ is \emph{proper}. If the support of $f\in C^\infty_c(G)$ can be written as $\bigsqcup_i S_i$ for finitely many compact open subsets $S_i\subset G$ such that $f|_{S_i}$ is a constant function with value $\alpha_i\ne0$, then we have 
    \[O^G_\gamma(f) = \sum_i \alpha_i\cdot \vol\big(c_\gamma^{-1}(S_i)\big),\]
    where $c_\gamma^{-1}(S_i)$ has finite positive volume as it is compact open in $G_\gamma\backslash G$.
\end{proof}

The following theorem is the main result of this paper.
\begin{thm}\label{th:main}
    Let $G = \Urm(V,h)$ be a unitary group over a non-archimedean local field $F$. Then for any $\gamma$ and $f\in C^\infty_c(G)$, the orbital integral $O^G_\gamma(f)$ absolutely converges.
\end{thm}

\begin{rmk}\label{rmk:history}
    When $F$ is a non-archimedean local field of characteristic~$0$ and $G$ is the group of $F$-points of a connected reductive group, the absolute convergence of orbital integrals on $G$ was established by Rao \cite{RangaRao:OrbInt}. The main idea is to use the Jordan decomposition of $\gamma$ to reduce the convergence of $O^G_\gamma(f)$ to that of unipotent orbital integrals, and then apply the exponential map to further reduce to nilpotent orbital integrals on the Lie algebra. If the characteristic of $F$ avoids a finite list of bad primes depending on $G$, this strategy can be adapted to handle elements $\gamma\in G$ admitting a rational Jordan decomposition (see \cite[Theorems~58, 61]{McNinch:NilpOrbits}).
    
    For inner forms of $\bGL_n$, the absolute convergence of orbital integrals was established without restriction on the characteristic of $F$ by Deligne--Kazhdan--Vign\'eras \cite[\S{A.1}]{DeligneKazhdanVigneras:CenSimAlg}, building on ideas of Howe \cite{Howe:FourierTransforms-Germs-GLn} (see also \cite[\S4.8]{Laumon:Book1}). Our proof of Theorem~\ref{th:main} is inspired by a synthesis of both approaches.
\end{rmk}
The remainder of the paper is devoted to proving this theorem. We begin by reducing the proof to a simpler case in this section.

\begin{lem}\label{lem:Gs}
    Let $G$ be as in \eqref{eq:prod-group}, and fix $\gamma\in G$. Suppose that there exists a semisimple element $s$ in the centre of $G_\gamma$. In particular, $G_s$ is unimodular and it contains $G_\gamma$. We fix Haar measures on $G$, $G_s$, and $G_\gamma$.

    For $f\in C^\infty_c(G)$, we define $f^{(s)}\in C^\infty_c(G_s)$ as follows
    \[
    f^{(s)}(x)\coloneqq \int_{G_s\backslash G} f(g^{-1}xsg)\frac{dg}{dg_s}.
    \]
    If $O^{G_s}_{\gamma s^{-1}}(f^{(s)})$ converges, then so does $O^G_\gamma(f)$ and we have $O^G_\gamma(f) = O^{G_s}_{\gamma s^{-1}}(f^{(s)})$.
\end{lem}
\begin{proof}
    We claim that $(G_s)_{\gamma s^{-1}} = G_\gamma$. In fact, for $x\in G_s$ we have $x^{-1}(\gamma s^{-1})x = \gamma s^{-1}$ if and only if $x^{-1}\gamma x = \gamma$. Thus,
    \begin{align*}
        O^{G_s}_{\gamma s^{-1}}(f^{(s)}) 
        &= \int_{G_\gamma\backslash G_s} f^{(s)}(x^{-1}(\gamma s^{-1})x) \frac{dx}{dg_\gamma}\\
        &=\int_{G_\gamma\backslash G_s} \int_{G_s\backslash G} f(g^{-1}x^{-1}\gamma xg) \frac{dg}{dx}\frac{dx}{dg_\gamma}\\
        &= \int_{G_\gamma\backslash G}f(g^{-1}\gamma g)\frac{dg}{dg_\gamma} = O^G_\gamma(f),
    \end{align*}
    where from the second line to the third, we renamed $xg$ as $g$.
\end{proof}

If $\gamma$ admits a Jordan decomposition, we apply Lemma~\ref{lem:Gs} by setting $s$ to be the semisimple part of $\gamma$, and reduce the convergence of $O^G_\gamma(f)$ to that of unipotent orbital integral. (See Remark~\ref{rmk:history} for a related remark.) For any element $\gamma\in G=\Urm(V,h)$ without Jordan decomposition, we can still choose a semisimple element $s\in G$ satisfying the conditions of Lemma~\ref{lem:Gs} so that $(G_s,\gamma s^{-1})$ is ``simpler'' than $(G,\gamma)$, which we explain now.

\begin{subequations}\label{eq:red-to-primary}
We view $V$ as an $E[T]$-module via the action of $\gamma$, and adopt the notation from \eqref{eq:elem-div} and Corollary~\ref{cor:elem-div-self-dual}, including $\Sigma=\Sigma_h\sqcup\Sigma_{nd}$, $\wp_a$, and $(V_a,h_a)$. For each $a\in \Sigma_{nd}$, we choose $t_a\in E^\times$ with $N_{E/F}(t_a)=1$ that are pairwise distinct. Similarly, for any $a\in\Sigma_h$, we choose $t_a\in F^\times$ that are pairwise distinct.

Now we define 
\begin{equation}\label{eq:red-to-primary:s}
    s\coloneqq \bigoplus_{a\in\Sigma_{nd}} \Big(t_a\cdot \id_{V_a}\Big)\oplus\bigoplus_{a\in\Sigma_h}\Big( t_a\cdot\id_{V\{\wp_a\}}\oplus t_a^{-1}\cdot \id_{V\{\overline\wp_a^\vee\}} \Big)\in \GL_E(V).
\end{equation}
Clearly, $s$ is in the centre of the commutant
\[C_\gamma\cong \prod_{a\in\Sigma_{nd}}\End_{E[T]}(V_a) \times \prod_{a\in\Sigma_{h}}\left( \End_{E[T]}\big(V\{\wp_a\} \big) \times \End_{E[T]}\big(V\{\overline\wp_a^\vee\} \big) \right),\]
so $s$ is in the centre of $\Gbf_\gamma(E)$. Furthermore, we clearly have $s^\star = s^{-1}$, so  $s\in G = \Urm(V,h)$. Thus, $s$ is in the centre of $G_\gamma$.

Using the orthogonal direct sum decomposition $(V,h) = \bigoplus_{a\in\Sigma}(V_a,h_a)$ one can deduce 
\begin{equation}
    G_s = \prod_{a\in \Sigma}G_{a},\qquad\text{where }G_{a}=
    \begin{cases}
        U(V_a,h_a) &\text{if }a\in\Sigma_{nd};\\
        \GL_E(V\{\wp\})& \text{if }a\in \Sigma_h.
    \end{cases}
\end{equation}
Since $\gamma\in G_\gamma\subset G_s$, we can write $\gamma = (\gamma_a)_{a\in\Sigma}$ for $\gamma_a\in G_a$. Thus, we have 
\begin{equation}
\gamma s^{-1} = \left( t_a^{-1}\cdot\gamma_a \right)_{a\in \Sigma}\in G_s.
\end{equation}

For any $a\in\Sigma_{nd}$ so $\wp_a=\overline\wp_a^\vee$, set $\wp_a'(T)\coloneqq t_a^{-d_{\wp_a}}\wp_a (t_a\cdot T)$, which is also monic and irreducible. Since $N_{E/F}(t_a)=1$, we have $\wp_a'=\overline\wp_a'^\vee$. Furthermore, the $E[T]$-module structure on $V_a$ defined by $t_a^{-1}\cdot\gamma_a$ is $\wp'_a$-primary. This leads to the following definition.
\end{subequations}

\begin{defn}\label{def:primary}
    Let $\wp\in E[T]$ be a monic irreducible polynomial such that $\wp=\overline\wp^\vee$. We say that $\gamma \in G=\Urm(V,h)$ is \emph{$\wp$-primary} if the $E[T]$-module structure on $V$ defined by $\gamma$ is $\wp$-primary; that is, $V = V\{\wp\}$. We say that $\gamma$ is \emph{primary} if it is $\wp$-primary for some $\wp$.
\end{defn}
Now, let $f\in C^\infty_c(G)$ and write $f^{(s)} = \sum_{j}\bigotimes_{a\in\Sigma}f_{a,j}$ for some $f_{a,j}\in C^\infty_c(G_a)$, Lemma~\ref{lem:Gs} implies the following:
\[
    O^G_\gamma(f) = \sum_j \prod_{a\in\Sigma}O^{G_a}_{t_a^{-1}\cdot\gamma_a}(f_{a,j}).
\]
If $a\in\Sigma_h$, then $O^{G_a}_{t_a^{-1}\cdot\gamma_a}(f_{a,j})$ absolutely converges since $G_a = \GL_E\left( V\{\wp_a\} \right)$ (see \cite[Proposition~(4.8.9)]{Laumon:Book1}). This proves the following statement.

\begin{cor}\label{cor:Gs}
    To prove Theorem~\ref{th:main}, it suffices to prove the absolute convergence of $O^G_\gamma(f)$ for \emph{primary} elements $\gamma\in G$.
\end{cor}

\section{Parabolic subgroup attached to a $\wp$-primary element}\label{sec:parabolic}

We continue with the notation from \S\ref{sec:review}, where $\Gbf = \bU(V,h)$ and $G=\Gbf(F)$. The goal of this section is to naturally associate, to each $\wp$-primary element $\gamma\in G$, an $F$-rational cocharacter $\lambda\colon\GG_m\to \Gbf$. This construction, in particular, defines a parabolic subgroup and a grading on its Lie algebra, which will be an important tool in the proof of the main theorem.

When $\gamma$ is unipotent, such a construction appears in \cite[\S{IV.2}]{SpringerSteinberg:ConjClasses} for classical groups including unitary groups, and we directly generalise this construction to primary elements $\gamma$. We also note that McNinch \cite[\S3, \S4]{McNinch:NilpOrbits} has given a geometric invariant theory (GIT) construction of such cocharacters for unipotent element and nilpotent Lie algebra element in a greater generality.

\begin{subequations}
    \addtocounter{equation}{-1}
We start by reviewing the construction of the parabolic subgroup associated to a ``primary'' element $\gamma\in\GL_E(V)$, following \cite[Lemma~4.8.4]{Laumon:Book1}.
\begin{lem}\label{lem:Laumon4.8.4}
    Fix $\gamma\in \GL_E(V)$ and endow $V$ with the $E[T]$-module structure defined by $\gamma$. Suppose that there is a monic irreducible polynomial $\wp\in E[T]$ such that
\begin{equation}\label{eq:Laumon4.8.4:primary}
    V\cong \bigoplus_{i=1}^s \left( E[T]/\wp^{m_i} \right)^{r_i}
\end{equation}
where we order the exponents as $m_1>\cdots>m_s>0$ and $r_i$ denotes the multiplicity.

Set $\Hbf\coloneqq \bGL_E(V)$, and let $\Pbf$ be the parabolic subgroup of $\Hbf$ stabilising the flag $\left( \ker(\wp^k\mid V) \right)_{k=1,\cdots,m_1}$, with $\Pbf=\Mbf\Nbf$ denoting the fixed Levi decomposition. Let $\overline\gamma\in \Mbf(E)$ be the image of $\gamma$ under the Levi quotient. Then, the following properties hold.
\begin{enumerate}
    \item\label{lem:Laumon4.8.4:dim} $\Pbf$ contains $\Hbf_\gamma$, and we have 
    \begin{equation}\label{eq:Laumon4.8.4:dim}
        \dim(\Hbf_\gamma)=\dim(\Mbf_{\overline\gamma})=d_\wp\cdot\sum_{i=1}^s (m_i-m_{i+1})\Big(\sum_{j=1}^ir_j\Big)^2,
    \end{equation}
    where $d_\wp\coloneqq\deg\wp$ and $m_{s+1}\coloneqq 0$.
    \item\label{lem:Laumon4.8.4:closure} $\overline\gamma\in \Mbf(E)$ is \emph{closed}, and the closure of $\Ocal_{\Pbf(E)}(\gamma)$ in $\Pbf(E)$ for the analytic topology is $\Ocal_{\Mbf(E)}(\overline\gamma)\cdot \Nbf(E)$.
\end{enumerate}
\end{lem}
\begin{proof}
    This lemma is immediate from \cite[Lemma~4.8.4]{Laumon:Book1} and its proof, except the formula for $\dim(\Mbf_\gamma)$ in \eqref{eq:Laumon4.8.4:dim}. Let $E_\wp\coloneqq E[T]/\wp$, and set 
    \[\gr_\wp^k V\coloneqq \ker(\wp^k\mid V)/\ker(\wp^{k-1}\mid V)\qquad\text{for any }k\geqslant1.\]
    Then, $\gr_\wp^k(V)$ has a natural $E_\wp$-vector space structure, and one can check
    \begin{equation}\label{eq:Laumon4.8.4:centraliser}
        \Mbf = \prod_k\bGL_E(\gr_\wp^k V) \quad\text{and}\quad \Mbf_{\overline\gamma} = \bRes_{E_\wp/E}\left( \prod_k\bGL_{E_p}(\gr_\wp^kV) \right).
    \end{equation}
    Furthermore, we can compute the $E_\wp$-dimension of $\gr^kV$ as follows:
    \begin{equation}\label{eq:Laumon4.8.4:grading}
        \dim_{E_\wp}\left( \gr_\wp^k V \right)=\sum_{j=1}^ir_j \quad\text{if }m_{i+1}<k\leqslant m_i.
    \end{equation}
    Now, the formula for $\dim(\Mbf_{\overline\gamma})$ immediately follows.
\end{proof}
\end{subequations}

\begin{rmk}\label{rmk:dim-eq}
In the proof of \cite[Lemma~4.8.4]{Laumon:Book1}, it is also shown that the dimension equality $\dim(\Hbf_\gamma) = \dim(\Mbf_{\overline\gamma})$ implies the openness of $\Ocal_{\Pbf(E)}(\gamma)$ in $\Ocal_{\Mbf(E)}(\overline\gamma)\cdot \Nbf(E)$. We will later repeat this proof for the unitary group, and also deduce a result when the dimension equality does not hold. See Proposition~\ref{prop:openness-levi-preimage} for details.
\end{rmk}

\begin{rmk}\label{rmk:rationality-parabolic} 
    Suppose that  $\gamma$ is a unipotent element in a unitary group $G=\Urm(V,h)$, viewed as an element of $\GL_E(V)$. Then Lemma~\ref{lem:Laumon4.8.4} associates to $\gamma$ a parabolic subgroup $\Pbf\subset \Gbf_E$ defined over $E$. We note that $\Pbf$ may \emph{not} be defined over $F$. For example, let $(V,h)$ be a rank-$3$ non-degenerate hermitian $E$-space, and choose $\gamma\in G$ so that the corresponding $E[T]$-module $V$ is isomorphic to $E[T]/(T-1)^2\oplus E[T]/(T-1)$. (To construct such $\gamma$, choose a rank-$2$ split hermitian $E$-subspace $(V_0,h_0)$ of $(V,h)$, and let $\gamma$ be the image of a regular unipotent element $\gamma_0\in \Urm(V_0,h_0)$.) Then $\Pbf\subset \Gbf_E$ is a parabolic subgroup that preserves a $2$-dimensional subspace of $V$. On the other hand, $\Gbf=\Ubf(V,h)$ is quasi-split with $F$-semisimple rank~$1$, so any proper parabolic subgroup is a Borel subgroup. In particular, $\Pbf$ is not defined over $F$.
\end{rmk}

\begin{subequations}
For the rest of the section, we will associate to a primary element $\gamma\in G=\Urm(V,h)$ an \emph{$F$-rational} parabolic subgroup of $\Gbf=\Urm(V,h)$, which satisfy slightly weaker conditions than Lemma~\ref{lem:Laumon4.8.4}. When $\gamma$ is unipotent, our construction is consistent with \cite[\S{IV.2.22}]{SpringerSteinberg:ConjClasses}. To explain, let us review parabolic subgroups of a unitary group. 

\addtocounter{equation}{-1}
\begin{constr}\label{constr:unitary-parabolic}
To specify an $F$-rational parabolic subgroup $\Pbf$ of $\Gbf = \bU(V,h)$, it is equivalent to give an \emph{isotropic flag} $\Fil^\bullet V$ in $V$. Indeed, the stabiliser of an isotropic flag is an $F$-rational parabolic, and conversely, any $F$-rational parabolic fixes a unique isotropic flag. Throughout, we index isotropic filtrations by positive integers in decreasing order:
\begin{equation}
    \Fil^1 V \supseteq \Fil^2 V \supseteq \cdots \supseteq \Fil^{m-1} V \supsetneq \Fil^{m} V=0.
\end{equation}
(We allow $\Fil^j V = \Fil^{j+1} V$.) We set $\Fil^j V=0$ if $j\geqslant m$. 

Given an isotropic flag $\left( \Fil^j V \right)_{j \geq 1}$ in $V$, we can extend it to a complete flag indexed by all integers by setting
\begin{equation}\label{eq:full-flag}
    \Fil^{-j}V \coloneqq \left( \Fil^{j+1}V \right)^\perp\qquad \forall j\geqslant0.
\end{equation}
If $\Pbf \subset \Gbf$ is the $F$-rational parabolic subgroup stabilising the isotropic flag $\left( \Fil^j V \right)_{j \geq 1}$, then its base change $\Pbf_E \subset \Gbf_E \cong \bGL_E(V)$ is the stabiliser of the full flag $\left( \Fil^j V \right)_{j \in \ZZ}$.

The choice of a Levi decomposition $\Pbf = \Mbf \ltimes \Nbf$ (or equivalently, a Levi decomposition of $\Pbf_E$ defined over $F$) is equivalent to specifying a splitting of the full flag $\left( \Fil^j V \right)_{j \in \ZZ}$; that is,
\begin{equation}\label{eq:splitting}
    V = \bigoplus_{j\in \ZZ} V(j),
\end{equation}
such that $\hat h\left( V(-j) \right) = \overline{V(j)}^\vee$, where $\hat h$ is as in \eqref{eq:h-hat}. The choice of such a splitting is, in turn, determined by the following data:
\begin{enumerate}
    \item an orthogonal decomposition $(V,h)=\left( V(0), h_0 \right)\oplus \left( V', h' \right)$, where $(V',h')$ is a split hermitian $E$-space with Lagrangian $\Fil^1 V$;
    \item a splitting $\Fil^1 V = \bigoplus_{j\geqslant1} V(j)$ of the filtration $\left( \Fil^j V \right)_{j\geqslant1}$.
\end{enumerate}
Note that $V(0)$ is allowed to be zero. Now, the Levi subgroup $\Mbf$ corresponding to the splitting $V = \bigoplus_jV(j)$ is given as follows
\begin{equation}\label{eq:Levi-Unitary}
    \Mbf = \bU(V(0),h_0) \times\prod_{j\geqslant1}\bGL_E\left( V(j) \right) ,
\end{equation}
where $\bU(V(0),h_0)$ is trivial if $V(0)=0$.

Equivalently, the choice of splitting $V = \bigoplus_{j \in \ZZ} V(j)$ can be encoded by a cocharacter $\lambda \colon \GG_m \to \Gbf$ defined by
\begin{equation}\label{eq:lambda}
    \forall t \in \GG_m(F), \qquad \lambda(t) \coloneqq \bigoplus_{j \in \ZZ} (t^j \id_{V(j)}).
\end{equation}
Indeed, the direct sum decomposition $V = \bigoplus_{j \in \ZZ} V(j)$ is equivalent to the weight space decomposition for $\lambda$ with respect to the $\GG_m$-action given by $\lambda_E\colon (\GG_m)_E \to \Gbf_E \cong \bGL_E(V)$. The symmetry imposed on the splitting is equivalent to $\lambda$ being defined over $F$.

Given a cocharacter $\lambda\colon\GG_m\to\Gbf$ over $F$, we obtain the \emph{weight space decomposition} of the Lie algebra $\gfr = \bigoplus_{j\in\ZZ}\gfr(j)$, where
\begin{equation}\label{eq:wt-Lie}
    \gfr(j) = \{X\in \gfr\mid \Ad(\lambda(t))(X)=t^jX\ \forall t\in \GG_m(F)\}.
\end{equation}
We may write $\gfr(j;\lambda)$ for $\gfr(j)$ if $\lambda$ needs to be specified.
\end{constr}
\end{subequations}

\begin{subequations}
    \addtocounter{equation}{-1}
We record some basic properties of the decomposition \eqref{eq:wt-Lie}.
\begin{lem}\label{lem:nfr-wt-decomp}
    \begin{enumerate}
        \item We have $\mfr = \gfr(0)$, $\pfr = \bigoplus_{j\geqslant 0} \gfr(j)$, and $\nfr = \bigoplus_{j>0}\gfr(j)$.
        \item For any $j, j' \geqslant 0$, we have $[\gfr(j), \gfr(j')] \subseteq \gfr(j + j')$. In particular, for any $j \geqslant 0$, the subspace
            \begin{equation}
                \nfr_j \coloneqq \bigoplus_{k \geqslant j} \gfr(k)
            \end{equation}
            is an ideal of $\pfr$, and each quotient $\nfr_j / \nfr_{j+1}$ is abelian.
        \item\label{lem:nfr-wt-decomp:odd-wt} For any $j \geqslant 1$, we have $\gfr(j)_E\cong \bigoplus_{k \in \ZZ} \Hom_E\left(V(k), V(k + j)\right)$. Furthermore, if $j$ is \emph{odd} then we have the following commutative diagram
        \begin{equation}\label{eq:nfr-wt-decomp:odd-wt}
        \begin{tikzcd}[row sep=large]
            \gfr(j) \arrow[d, hook] \arrow[r, "\cong"] & \bigoplus\limits_{k\geqslant -(j-1)/2}\Hom_E\left( V(k),V(k+j) \right); \\
            \gfr(j)_E \arrow[r, "\cong"] & \bigoplus\limits_{k\in\ZZ}\Hom_E\left( V(k),V(k+j) \right) \arrow[u, two heads, "\pr"']
        \end{tikzcd}
    \end{equation}
        \item\label{lem:nfr-wt-decomp:exp} There exists an $F$-subgroup $\Nbf_2 \subseteq \Nbf$ defined by
            \begin{equation}\label{eq:N2}
                \Nbf_2(E) = \left\{ g \in \Nbf(E) \ \middle| \ (g - 1)\left(\Fil^j V\right) \subseteq \Fil^{j+2} V \text{ for all } j \in \ZZ \right\},
            \end{equation}
            which is normal in $\Pbf$ with $\Lie(\Nbf_2)=\nfr_2$.
            
            Moreover, the following isomorphism over $E$
            \begin{equation}\label{eq:nfr-wt-decomp:mod-N2-over-E}
                \begin{tikzcd}[column sep=small]
                     \Pbf_E/\Nbf_{2,E} \arrow[rr, "\cong"] & & \Mbf_E\ltimes \gfr(1)_E ; &
                     m \bar{n} \arrow[r, maps to] & (m, \bar{n}-1) 
                \end{tikzcd}
            \end{equation}
            descends to an isomorphism $\Pbf / \Nbf_2 \cong \Mbf \ltimes \gfr(1)$ over $F$, where the semidirect product is with respect to the adjoint action.
    \end{enumerate}
\end{lem}
\begin{proof}
    The first three statements, except for \eqref{eq:nfr-wt-decomp:odd-wt}, can be verified directly or by base changing to $E$ and using standard properties of $\gfr_E \cong \gl_E(V)$. To verify \eqref{eq:nfr-wt-decomp:odd-wt}, recall that $\gfr \subset \gl_E(V)$ is the annihilator of the hermitian form $h$; that is,
    \begin{equation}
        h(\xi v,w) + h(v,\xi w)=0\qquad \forall v,w\in V;
    \end{equation}
    or equivalently, $\overline\xi^t J + J\xi = 0$ using the notation from \eqref{eq:star-formula}. Thus, for any $\xi_k\in \Hom_E(V(k),V(k+j))$ with $j\ne 2k$, there exists a unique  $\xi_{-k-j}\in \Hom_E(V(-k-j),V(-k))$ such that $\xi_k+\xi_{-k-j}\in \gfr$. If $j$ is an odd positive integer, then \eqref{eq:nfr-wt-decomp:odd-wt} immediately follows since $j\ne 2k$ for any $k$.
    
    For (\ref{lem:nfr-wt-decomp:exp}), observe that the condition \eqref{eq:N2} defines an $E$-subgroup of $\Nbf_E$ that is stable under the adjoint involution $\star$, and hence it descends to an $F$-subgroup $\Nbf_2 \subset \Nbf$. The normality in $P$ can be verified after base change to $E$, which is immediate. Lastly, to show that the isomorphism \eqref{eq:nfr-wt-decomp:mod-N2-over-E} is defined over $F$, it suffices to show that it the isomorphism 
    \begin{equation}\label{eq:nfr-wt-decomp:exp-gfr-1}
        \begin{tikzcd}[column sep=small]
            \gfr(1)_E \arrow[rr, "\cong"] & & \Nbf(E)/\Nbf_2(E) ; & \xi \arrow[r, maps to]& \overline{(1+\xi)}
        \end{tikzcd}
    \end{equation}
    restricts to $\gfr(1) \riso N/N_2$. To show this, choose $\xi\in \Hom_E(V(k),V(k+1))$ for some $k\geqslant0$, and let $\xi'\in \Hom_E(V(-k-1),V(-k))$ be the unique element with $\xi+\xi'\in\gfr(1)$. If $k\ne 0$, then one can directly show that $1+\xi+\xi' \in N$. If $k=0$ then there exists $\eta\in\Hom_E(V(-1),V(1))$ such that 
    \[1+\xi+\xi'+\eta=    
    \begin{pmatrix}
    \ddots & & & & \\
        &1_{V(1)} & \xi & \eta & \\
    &  & 1_{V(0)} & \xi'& \\
    &  &   & 1_{V(-1)}& \\
    & & & & \ddots\\
    \end{pmatrix}
    \]
    is in $N$.
    Indeed, if $\eta$ satisfies $h(\eta v,w)+h(v,\eta w) + h(\xi' v, \xi' w) = 0$ for any $v,w\in V(-1)$, then the above matrix is in $N$. 
\end{proof}
\end{subequations}

\begin{rmk}
    Although the map $\GL_E(V)\to\gl_E(V)$, sending $g\mapsto g-1$, induces an isomorphism from the unipotent subvariety to the nilpotent cone, it may not send a unipotent element of $G$ to an element in $\gfr$. The proof of Lemma~\ref{lem:nfr-wt-decomp}(\ref{lem:nfr-wt-decomp:exp}) works out because $\Nbf_2$ contains the derived subgroup of $\Nbf$.
\end{rmk}
To define an isotropic filtration, we need the following lemma.
\begin{lem}\label{lem:Fil1}
    Fix $\gamma\in G$, and view $V$ as an $E[T]$-module via $\gamma$. Fix a monic irreducible $\wp\in E[T]$ with $\wp=\overline\wp^\vee$, and set 
    \[
    W\coloneqq\sum_{k\geqslant1} \left( \ker(\wp^k)\cap \im(\wp^k) \right) \subset V.
    \]
    Then $W$ is totally isotropic.
\end{lem}
\begin{proof}
    It suffices to show that for any $v\in \ker(\wp^k)\cap \im(\wp^k)$ and $v'\in \ker(\wp^{k'})\cap \im(\wp^{k'})$, we have $h(v,v')=0$. We may assume $k\geqslant k'$ without loss of generality, and write $v = \wp(\gamma)^kw$ for some $w\in V$. Then we have
    \[
        h(\wp(\gamma)^kw,v')  =  h(w, \overline\wp(\gamma^{-1})^kv')), 
    \]
    since $\gamma^\star = \gamma^{-1}$. 
    
    Now, note that $\overline\wp(\gamma^{-1})^k v'=0$, since we have 
    \[\overline\wp(\gamma^{-1}) = \overline\wp(0)\gamma^{-d_\wp}\cdot\overline\wp^\vee(\gamma) = \overline\wp(0)\gamma^{-d_\wp}\cdot\wp(\gamma)\]
    and $v'\in \ker(\wp^{k'})\subseteq \ker(\wp^k)$. This shows $h(v,v') = h(w,0)=0$, as desired.
\end{proof}

\begin{defn}\label{def:assoc-fil}
    Let $\gamma\in G$ be a $\wp$-primary element for a monic irreducible $\wp\in E[T]$ with $\wp=\overline\wp^\vee$, and endow $V$ with the $E[T]$-module structure defined by $\gamma$. For any $j\geqslant 1$, set
    $$ \Fil^j V\coloneqq \sum_{k\geqslant 1}\left( \ker(\wp^k)\cap \im(\wp^{k+j-1}) \right). $$
    This defines an \emph{isotropic flag}, since $\Fil^1 V = W$ is totally isotropic by Lemma~\ref{lem:Fil1}.

    We extend it to a full filtration $\left( \Fil^j V \right)_{j\in \ZZ}$ by the recipe \eqref{eq:full-flag}, and \emph{choose} a splitting $V = \bigoplus_{j\in\ZZ}V(j)$ as explained in the paragraph below \eqref{eq:splitting}. This choice defines an $F$-rational cocharacter $\lambda\colon \GG_m\to \Gbf$ via \eqref{eq:lambda}, a parabolic subgroup with Levi decomposition $\Pbf = \Mbf\ltimes\Nbf$, and the grading $\gfr = \bigoplus_j\gfr(j)$ as in \eqref{eq:wt-Lie}.
\end{defn}

\begin{rmk}\label{rmk:parabolic-regular-case}
    Suppose in addition that $\gamma\in G$ is \emph{regular}; that is, $V\cong E[T]/\wp^m$. Then the base change $\Pbf_E$ of the parabolic subgroup $\Pbf$ constructed in Definition~\ref{def:assoc-fil} coincides with the parabolic subgroup of $\Gbf_E$ described in Lemma~\ref{lem:Laumon4.8.4}.
\end{rmk}

The following property is one of the main motivations for this construction
\begin{lem}\label{lem:centraliser-parabolic}
     In the setting of Definition~\ref{def:assoc-fil}, $\Pbf$ contains $\Gbf_\gamma$. 
\end{lem}
\begin{proof}
    Note that the commutant $C_\gamma = \End_{E[T]}(V)$ preserves $\ker(\wp^k)$ and $\im(\wp^{k+j-1})$, so it preserves the isotropic flag. Hence, we have $\Gbf_\gamma\subset \Pbf$. 
\end{proof}

\begin{subequations}
    \addtocounter{equation}{-1}
\begin{exa}\label{ex:assoc-fil}
Suppose that $V\cong \left( E[T]/\wp^m \right)^r$. Then for any $j\geqslant1$, we have 
\begin{equation}
    \Fil^j V = 
    \ker(\wp^k) ,\quad \text{where }k = \left\lfloor \frac{m-j+1}{2} \right\rfloor.
\end{equation}
That is, $\Fil^{m-1} V = \Fil^{m-2} V = \ker(\wp)$, $\Fil^{m-3} V = \Fil^{m-4} V = \ker(\wp^2)$, and so on.

The associated graded pieces for $1\leqslant j \leqslant m-1$ are given by
\begin{equation}
    \gr^j V = 
    \begin{cases}
        \ker(\wp^k)/\ker(\wp^{k-1}) & \text{if } k = \frac{m-j+1}{2} \text{ is an integer};\\
        0 & \text{otherwise}.
    \end{cases}
\end{equation}
That is, $\gr^j V$ is nonzero only when $j \equiv m-1 \pmod{2}$.

More generally, the filtration and the associated grading on  $V \cong \bigoplus_i \left( E[T]/\wp^{m_i} \right)^{r_i}$ are given by the direct sum of the filtrations and gradings on each summand.
\end{exa}
\end{subequations}

\begin{lem}\label{lem:assoc-fil}
    In the setting of Definition~\ref{def:assoc-fil}, write $V \cong \bigoplus_{i=1}^s \left( E[T]/\wp^{m_i} \right)^{r_i}$, where $m_1 > \cdots > m_s > 0$ and $r_i > 0$.
    \begin{enumerate}
        \item\label{lem:assoc-fil:full} For all $j \in \ZZ$, we have
            \[
            \Fil^{-j} V = \left( \Fil^{j+1} V \right)^\perp = \sum_{k \geqslant 1} \left( \ker(\wp^k) \cap \im(\wp^{k-j-1}) \right).
            \]
            In other words, the formula for $\Fil^j V$ extends naturally to all $j \in \ZZ$.
        \item\label{lem:assoc-fil:gr-dim} Set $I \coloneqq \{1, \ldots, s\}$, and partition it as follows:
            \[
                I_+ \coloneqq \{ i \in I \mid m_i \text{ is even} \}, \qquad
                I_- \coloneqq \{ i \in I \mid m_i \text{ is odd} \}.
            \]
            Then, for any $j \in \ZZ$, the graded piece $\gr^j V$ has a natural structure of an $E_\wp$-vector space, where $E_\wp \coloneqq E[T]/\wp$, and its $E_\wp$-dimension is given by:
            \[
                \dim_{E_\wp} \gr^j V =
                \begin{cases}
                    \sum\limits_{\substack{i \in I_-,\\ m_i \geqslant |j| + 1}} r_i & \text{if } j \text{ is even}, \\[2ex]
                    \sum\limits_{\substack{i \in I_+,\\ m_i \geqslant |j| + 1}} r_i & \text{if } j \text{ is odd}.
                \end{cases}
            \]
    \end{enumerate}
\end{lem}
\begin{proof}
    To prove (\ref{lem:assoc-fil:full}), it suffices to show that
    \[
    \left( \Fil^{j+1} V \right)^\perp \supseteq \sum_{k\geqslant 1}\left( \ker(\wp^k)\cap \im(\wp^{k-j-1}) \right),
    \]
    since the dimension of successive quotients of both sides match. Indeed, the full filtration $\left( \Fil^j V \right)$ by $\Fil^{-j}V\left( \Fil^{j+1} V \right)^\perp$ satisfies $\dim_E \gr^j V = \dim_E \gr^{-j}V$, while the same symmetry is satisfied by the the filtration defined by the right-hand side, which can be checked summand by summand.

    To check the above inclusion, we need to show that $h(v,v')=0$ for any $v\in \ker(\wp^k)\cap\im(\wp^{k+j})$ and $v'\in \ker(\wp^{k'})\cap \im(\wp^{k'-j-1})$. By assumption, we have
    \begin{itemize}
        \item $\exists w\in V$ such that $v=\wp(\gamma)^{k+j}w$ and $\wp(\gamma)^{2k+j}w=0$; and
        \item $\exists w'\in V$ such that $v'=\wp(\gamma)^{k'-j-1}w'$ and $\wp(\gamma)^{2k'-j-1}w'=0$.
    \end{itemize}
    As in the proof of Lemma~\ref{lem:Fil1}, it suffices to show that either $w$ or $w'$ is annihilated by $\wp(\gamma)^{k+k'-1}$. If $k+k'-1 \geqslant 2k'-j-1$, then $\wp(\gamma)^{k+k'-1} w' = 0$. Otherwise, $k' \geqslant k + j + 1$, which again implies $k+k'-1 \geqslant 2k+j$. This proves the desired inclusion, and hence claim~(\ref{lem:assoc-fil:full}).

    For (\ref{lem:assoc-fil:gr-dim}), observe that for each summand $E[T]/\wp^{m_i}$, the filtration $\Fil^j V$ is nonzero only for $j \leq m_i-1$, and the associated graded pieces are nonzero only for $j$ of the same parity as $m_i-1$. The dimension formula then follows by summing over all $i$ with the appropriate parity and range.
\end{proof}

\begin{subequations}
    \addtocounter{equation}{-1}
Let us record useful consequences of Lemma~\ref{lem:assoc-fil}.
\begin{cor}\label{cor:assoc-fil}
    In the setting of Definition~\ref{def:assoc-fil}, the following properties hold:
    \begin{enumerate}
        \item\label{cor:assoc-fil:Levi} The image $\overline\gamma \in M$ of $\gamma$ under the Levi quotient is closed, both as an element of $M$ and as an element of $\Mbf(E)$.
        \item \label{cor:assoc-fil:dim} We have
        \begin{equation}\label{eq:centraliser-dim}
             \dim (\Mbf_{\overline\gamma}) = d_\wp\sum_{i=1}^s(m_i-m_{i+1})\left( \Big(\sum_{i'\leqslant i;\ i' \in I_-} r_{i'} \Big)^2+ \Big(\sum_{i'\leqslant i;\ i' \in I_+} r_{i'} \Big)^2 \right)
        \end{equation}
        \item\label{cor:assoc-fil:N2} The image of $\gamma$ in $P/N_2 \cong M\ltimes \gfr(1)$ coincides with $\overline\gamma\in M$, where $\Nbf_2$ is defined in \eqref{eq:N2}.
    \end{enumerate}
\end{cor}
\begin{proof}
    We first prove (\ref{cor:assoc-fil:Levi}). Recall from \eqref{eq:Levi-Unitary} that $M = \Urm(\gr^0 V, h_0) \times \prod_{j\geqslant1}\GL_E\left( \gr^j V \right)$, and $\Mbf(E) = \prod_{j\in\ZZ}\GL_E\left( \gr^j V \right)$. Observe that the action of $\overline\gamma$ on $\gr^j V$ has minimal polynomial $\wp$ for each $j$. (This follows from the definition of the isotropic flag $\left( \Fil^j V \right)_{j\geqslant1}$ for positive $j$, and from Lemma~\ref{lem:assoc-fil} for $j\leqslant0$). By Lemma~\ref{lem:closed}(\ref{lem:closed:GL}), the image of $\overline\gamma$ in $\GL_E\left( \gr^j V \right)$ is closed. Consequently, $\overline\gamma$ is closed as an element of $\Mbf(E)$. Furthermore, the inclusion $M \hookrightarrow \Mbf(E)$ restricts to the identity map on $\GL_E\left( \gr^j V \right)$ for $j\geqslant1$ and to the natural inclusion $\Urm(\gr^0 V, h_0) \hookrightarrow \GL_E\left( \gr^0 V \right)$ at $j=0$. Thus, $\overline\gamma$ is also a closed element of $M$ by Lemma~\ref{lem:closed}(\ref{lem:closed:unitary}), completing the proof of (\ref{cor:assoc-fil:Levi}).

    To prove (\ref{cor:assoc-fil:dim}), note that the minimal polynomial of $\overline\gamma$ on $\gr^j(V)$ is $\wp$ for each $j\in\ZZ$, so we have 
    \begin{equation}
        \Mbf_{E,\overline\gamma} \cong \Res_{E_\wp/E}\Big(  \prod_{j\in\ZZ}\GL_{E_\wp}\left( \gr^j V \right)\Big).
    \end{equation}
    Therefore, its dimension is $[E_\wp:E]\cdot\sum_{j\in\ZZ}\left( \dim_{E_{\wp}}\left( \gr^j V \right) \right)^2$, which yields (\ref{cor:assoc-fil:dim}) by Lemma~\ref{lem:assoc-fil}(\ref{lem:assoc-fil:gr-dim}).
    
    Let us prove (\ref{cor:assoc-fil:N2}). Using the fixed Levi decomposition $\Pbf=\Mbf\ltimes\Nbf$ (or equivalently, the fixed splitting of the filtration $V = \bigoplus V(j)$), it suffices to show that $\overline\gamma^{-1}\gamma \in \Nbf_2(E)$, since $\overline\gamma^{-1}\gamma\in N$ and $N_2=N\cap \Nbf_2(E)$. Because $\Nbf_2(E)$ is a unipotent subgroup of a standard parabolic subgroup in $\GL_E(V)$, this is equivalent to showing that the corresponding endomorphism satisfies
    \begin{equation}
        \overline\gamma^{-1}\gamma - 1 = \overline\gamma^{-1}(\gamma-\overline\gamma) \in \nfr_{2,E}.
    \end{equation}
    We identify $\nfr_{2,E}$ with the space of nilpotent endomorphisms of $V$ that map $\Fil^j V$ into $\Fil^{j+2} V$ for all $j$. Because $\overline\gamma$ is an automorphism that strictly preserves each graded piece $V(j)$, it suffices to prove that $(\gamma-\overline\gamma)\left( \Fil^j V \right) \subseteq \Fil^{j+2} V$. 
    
    Crucially, the property of belonging to the ideal $\nfr_{2,E}$ is invariant under conjugation by any element $g \in \Pbf(E)$. Fix an $E[T]$-module isomorphism $V \cong \bigoplus_i \left( E[T]/\wp^{m_i} \right)^{r_i}$. Any two splittings of the same filtration are conjugate by an element of the unipotent radical of the parabolic in $\GL_E(V)$. Therefore, we may choose a conjugating element in $\Pbf(E)$ such that our splitting, and the corresponding action by $\overline\gamma$, precisely respects the $E[T]$-cyclic direct sum decomposition. 
    
    This conjugation allows us to treat $\gamma-\overline\gamma$ as a nilpotent endomorphism that stabilises each $E[T]$-cyclic summand. Within each isolated summand $E[T]/\wp^{m_i}$, the operator $\gamma-\overline\gamma$ acts by shifting the $\wp$-adic kernel filtration by exactly one step. As demonstrated in Example~\ref{ex:assoc-fil}, a one-step shift in the $\wp$-adic kernel filtration corresponds to a jump of exactly $2$ in the indexing of the isotropic flag. This yields $(\gamma-\overline\gamma)\left( \Fil^j V \right) \subseteq \Fil^{j+2} V$ on each summand, which proves the claim for $V$ and establishes (\ref{cor:assoc-fil:N2}).
\end{proof}
\end{subequations}
Lemma~\ref{lem:centraliser-parabolic} and Corollary~\ref{cor:assoc-fil} will play an important role in the proof of the main theorem. See \S\ref{sec:closure} for more details.
\begin{rmk}
    When $\gamma\in G$ is unipotent (i.e., $(T-1)$-primary), our construction of the associated parabolic subgroup recovers the classical unitary case in \cite[\S{IV.2.22}]{SpringerSteinberg:ConjClasses}. In that reference, the cocharacter $\lambda$ \eqref{eq:lambda}---or equivalently, the splitting of the isotropic flag---is constructed canonically. McNinch \cite{McNinch:NilpOrbits} further generalised this approach to arbitrary connected reductive groups in sufficiently large characteristic, and used it to prove the absolute convergence of unipotent orbital integrals under the same restriction on the characteristic.

    Let $\gamma$ be $\wp$-primary for some \emph{inseparable} monic irreducible polynomial $\wp$. Since such $\gamma$ does not admit a Jordan decomposition over $F$, it is not possible to use Lemma~\ref{lem:Gs} to reduce the convergence of orbital integrals to the unipotent part of $\gamma$. Furthermore, it does not seem possible to \emph{canonically} construct a cocharacter $\lambda$ \eqref{eq:lambda} in this setting. Instead, we bypass this by using the primary decomposition of the $E[T]$-module $V$ to canonically construct the isotropic flag in Construction~\ref{constr:unitary-parabolic}, and hence the parabolic subgroup.
%
\end{rmk}

\section{Closure of a primary conjugacy class in the associated parabolic}\label{sec:closure}

In \S\ref{sec:parabolic}, we associated an $F$-rational parabolic subgroup $\Pbf$ to any \emph{primary} element $\gamma\in G=\Urm(V,h)$. The goal of this section is to construct a closed $F$-submanifold $\Vcal(\gamma)\subset P$ that contains $\Ocal_P(\gamma)$ as an open $F$-submanifold. If $\gamma$ is unipotent (i.e., $(T-1)$-primary), then $\Vcal(\gamma)$ is a (Zariski closed) subgroup of $N$ (see Remark~\ref{rmk:unip-orbit}). In general, $\Vcal(\gamma)$ is a fibration of a subgroup of $N$ over $\Ocal_M(\overline\gamma)$.

The following technical result gives the \emph{expected dimension} of $\Vcal(\gamma)$.
\begin{subequations}
    \addtocounter{equation}{-1}
\begin{prop}\label{prop:openness-levi-preimage} 
    Fix $\gamma\in G$, and let $\Pbf\subset \Gbf$ be a parabolic subgroup with Levi decomposition $\Pbf = \Mbf\ltimes\Nbf$. Suppose that $\Pbf$ contains $\Gbf_\gamma$.
    \begin{enumerate}
    \item\label{prop:openness-levi-preimage:centraliser} Let $\overline\gamma\in M$ denote the image of $\gamma$ under the Levi quotient. Then we have
    \begin{equation}\label{eq:openness-levi-preimage:dim}
    \dim(\Gbf_\gamma)\geqslant\dim(\Mbf_{\overline\gamma}).
    \end{equation}
    Furthermore, the codimension of $\Ocal_P(\gamma)$ in $\Ocal_M(\overline\gamma) \cdot N$ as an $F$-submanifold is equal to $\dim(\Gbf_\gamma) - \dim(\Mbf_{\overline\gamma})$.
    \item\label{prop:openness-levi-preimage:closure}
    Suppose that the dimension inequality \eqref{eq:openness-levi-preimage:dim} is an equality. Then, $\Ocal_P(\gamma)$ is contained in $\Ocal_M(\overline\gamma)\cdot N$ as an \emph{open} submanifold.
\end{enumerate}
\end{prop}
Later, we will specialise to the case when $\gamma$ is primary and $\Pbf$ is the parabolic subgroup as in Definition~\ref{def:assoc-fil}. See Corollary~\ref{cor:openness-levi-preimage} and the discussions that follow.
\begin{proof}
The natural projection $P\twoheadrightarrow M$ maps $\Ocal_P(\gamma)$ to $\Ocal_M(\overline \gamma)$, so $\Ocal_P(\gamma)$ is contained in the full preimage of $\Ocal_M(\overline\gamma)$, which is $\Ocal_M(\overline\gamma)\cdot N$. 

Next, we describe the tangent spaces of $\Ocal_P(\gamma)$ and $\Ocal_M(\overline\gamma)\cdot N$ at $\gamma$, and show that the natural inclusion induces an injective map on the tangent spaces. 
For this, it is convenient to replace $c_\gamma\colon P\to P$ by the following translate
    \begin{equation}
        \gamma^{-1}c_\gamma\colon\begin{tikzcd}[column sep = small]
        P \arrow[r]& P    
    \end{tikzcd},
    \end{equation}
    which sends $g\in P$ to $[\gamma^{-1},g^{-1}] = \gamma^{-1}g^{-1}\gamma g$. This map preserves the identity element, and induces the following map on the Lie algebra $\pfr$ of $P$:
    \begin{equation}
        1-\Ad(\gamma^{-1})\colon \begin{tikzcd}[column sep = small]
        \pfr \arrow[r] &\pfr
        \end{tikzcd},
        \end{equation}
        where $\Ad\colon P\to \GL(\pfr)$ is the adjoint representation.
        
        Assume now that $\Gbf_\gamma\subset \Pbf$ as in (\ref{prop:openness-levi-preimage:centraliser}), so that $G_\gamma$ is also the centraliser of $\gamma$ in $P$. Let $\gfr_\gamma$, $\mfr$, $\nfr$, and $\mfr_{\overline\gamma}$ denote the Lie algebras of $G_\gamma$, $M$, $N$, and $M_{\overline\gamma}$, respectively. 

        The map $\gamma^{-1}c_\gamma$ induces a locally closed embedding $G_\gamma\backslash P\to P$, which is the restriction of an immersion of affine algebraic varieties. By smoothness of $\Gbf_\gamma$, the tangent space at the identity is $\gfr_\gamma\backslash\pfr$, and $\gamma^{-1}c_\gamma$ induces an injective map on tangent spaces at the identity. Thus, 
        \begin{equation}
        \gfr_\gamma = \ker (1-\Ad(\gamma^{-1})\mid \pfr).
        \end{equation}

    Since $\Mbf$ is a direct product involving general linear groups and a unitary group,
    $\Mbf_{\overline\gamma}$ is smooth and the map $\overline\gamma^{-1}c_{\overline\gamma}\colon M_{\overline\gamma}\backslash M\to M$ satisfies the analogous properties to $\gamma^{-1}c_\gamma\colon P_\gamma\backslash P\to P$, including 
    \begin{equation}
        \mfr_{\overline\gamma} = \ker(1-\Ad(\overline\gamma^{-1})\mid \mfr).
        \end{equation}
        Furthermore, it identifies the tangent space of $\Ocal_M(\overline\gamma)\cdot N$ at $\gamma$ with
        \begin{equation}
        \pfr \times_{\mfr} (\mfr_{\overline\gamma}\backslash\mfr),
        \end{equation}
        where the fibre product is taken with respect to the natural projection $\pfr\twoheadrightarrow\mfr$ and the map $1-\Ad(\overline\gamma^{-1})\colon \mfr_{\overline\gamma}\backslash\mfr \to\mfr$. 

         With these identifications, the natural inclusion  $\Ocal_P(\gamma)\hookrightarrow\Ocal_M(\overline\gamma)\cdot N$ induces the following map on tangent spaces at $\gamma$:
        \begin{equation}\label{eq:embedding-tang-sp}
        \begin{tikzcd}
        \gfr_\gamma\backslash\pfr \arrow[r]  & 
        \pfr \times_{\mfr} (\mfr_{\overline\gamma}\backslash\mfr)
        \end{tikzcd},
        \end{equation}
        sending the coset of $a\in\pfr$ to $((1-\Ad(\gamma^{-1}))(a), \overline a)$, where $\overline a$ is the image of $a$ in $\mfr_{\overline\gamma}\backslash\mfr$. This map is injective since $\gfr_\gamma = \ker (1-\Ad(\gamma^{-1})\mid \pfr)$. Comparing $F$-dimensions, we obtain 
        \begin{multline*}
             \dim\left( \Ocal_P(\gamma) \right)= \dim_F(\pfr) - \dim_F(\gfr_\gamma) \\
             \leqslant \dim_F(\pfr) - \dim_F(\mfr_{\overline\gamma}) = \dim\left( \Ocal_M(\overline\gamma)\cdot N \right),
        \end{multline*}
        which gives the dimension inequality \eqref{eq:openness-levi-preimage:dim} and the codimension formula.

        Now suppose that $\dim_F(\gfr_\gamma) = \dim_F(\mfr_{\overline\gamma})$. Then the $P$-equivariant inclusion $\Ocal_P(\gamma)\hookrightarrow\Ocal_M(\overline\gamma)\cdot N$ induces an isomorphism on tangent spaces at $\gamma$, hence everywhere. By \cite[Part~II, Chap~III, §9, Thm~2]{Serre:LieThy}, the embedding is a local isomorphism of analytic $F$-manifolds, so it is an open embedding.
\end{proof}
\end{subequations}

If $\gamma$ is primary, we have constructed a parabolic subgroup $\Pbf$ containing $\Gbf_\gamma$ (see Lemma~\ref{lem:centraliser-parabolic}), and have already computed the dimensions of $\Gbf_\gamma$ and $\Mbf_{\overline\gamma}$ (see \eqref{eq:Laumon4.8.4:dim} and \eqref{eq:centraliser-dim}). Under a favourable assumption on $\gamma$ (including \emph{regular} elements), we have $\dim(\Gbf_\gamma) = \dim (\Mbf_{\overline\gamma})$ so (\ref{prop:openness-levi-preimage:closure}) can be applied. However, the dimension equality may be impossible to arrange for general primary elements $\gamma$. See Remark~\ref{rmk:openness-levi-preimage} for further discussions.
We record another useful property below.
\begin{subequations}
    \addtocounter{equation}{-1}
\begin{lem}\label{lem:Laumon4.8.5}
    If $\Pbf$ contains $\Gbf_\gamma$, then we get a natural long exact sequence
\begin{equation}\label{eq:Laumon4.8.5}
    \begin{tikzcd}[column sep=small, row sep=small]
    0 \arrow[r] 
    & \ker\big(1-\Ad(\gamma^{-1})\mid\nfr)\big) \arrow[r] 
    & \gfr_\gamma \arrow[r] 
    & \mfr_{\overline\gamma} \arrow[dll, rounded corners,
        to path={
            -- ([xshift=2.5ex]\tikztostart.east)
            |- ($(\tikztostart.south)!0.5!(\tikztotarget.north)$)
            -| ([xshift=-2.5ex]\tikztotarget.west)
            -- (\tikztotarget)
        }] \\
    & \coker\big(1-\Ad(\gamma^{-1})\mid\nfr)\big) \arrow[r] 
    & \nfr_{[\gamma]} \arrow[r] 
    & 0
\end{tikzcd}
\end{equation}
    where $\nfr_{[\gamma]}\coloneqq\im\left( \nfr\to \coker\big(1-\Ad(\gamma^{-1})\mid\pfr\big) \right)$.
    In particular, we have a natural isomorphism of $1$-dimensional $F$-vector spaces
    \begin{equation}\label{eq:Laumon4.8.5:Haar}
        \begin{tikzcd}
        \det(\gfr_\gamma) \arrow[r,"\cong"] & \det(\mfr_{\overline\gamma})\otimes\det(\nfr_{[\gamma]})
    \end{tikzcd}.
    \end{equation}
    Furthermore, we have $\nfr_{[\gamma]}=0$ if and only if $\dim(\Gbf_\gamma)=\dim(\Mbf_{\overline\gamma})$, in which case we have a natural isomorphism $\det{\gfr_\gamma}\cong\det(\mfr_{\overline\gamma})$.
\end{lem}

\begin{proof}
    (Compare with the proof of Lemma~(4.8.4) and Remark~(4.8.5) in \cite{Laumon:Book1}.) The long exact sequence \eqref{eq:Laumon4.8.5} is a truncation of the snake lemma long exact sequence associated to the following diagram:
    \begin{equation}
        \begin{tikzcd}[column sep = large]
            0 \arrow[r] & \nfr \arrow[r] \arrow[d,"1-\Ad(\gamma^{-1})"] & \pfr \arrow[r] \arrow[d,"1-\Ad(\gamma^{-1})"] & \mfr \arrow[r] \arrow[d,"1-\Ad(\overline\gamma^{-1})"] & 0 \\
            0 \arrow[r] & \nfr \arrow[r] & \pfr \arrow[r] & \mfr \arrow[r] & 0
        \end{tikzcd}
    \end{equation}
    Note that $\nfr_{[\gamma]}$ coincides with the image of $\coker\big(1-\Ad(\gamma^{-1})\mid\nfr\big)$ in $\coker\big(1-\Ad(\gamma^{-1})\mid\pfr\big)$. The isomorphism \eqref{eq:Laumon4.8.5:Haar} follows formally from \eqref{eq:Laumon4.8.5}. Finally, considering the last three terms of the snake lemma sequence, we obtain
    \[
        \begin{tikzcd}[column sep=small]
        0 \arrow[r] & \nfr_{[\gamma]} \arrow[r] &
        \coker\big(1-\Ad(\gamma^{-1})\mid\pfr\big) \arrow[r,"(*)"] &
        \coker\big(1-\Ad(\overline\gamma^{-1})\mid\mfr\big) \arrow[r] & 0
        \end{tikzcd}
    \]
    and observe that $(*)$ is an isomorphism if and only if $\dim(\gfr_\gamma) = \dim(\mfr_{\overline\gamma})$. This completes the proof.
\end{proof}
\end{subequations}

\begin{rmk}\label{rmk:Laumon4.8.7} (Compare with \cite[Remark~4.8.7]{Laumon:Book1}.)
    If $P$ contains $\Gbf_\gamma$ and $\dim(\Gbf_\gamma) = \dim(\Mbf_{\overline\gamma})$ (so $\nfr_{[\gamma]}=0$), then \eqref{eq:Laumon4.8.5:Haar} defines a canonical bijections between the sets of Haar measures on $G_\gamma$ and $M_{\overline\gamma}$ in the setting of Proposition~\ref{prop:openness-levi-preimage}.  In fact, a Haar measure on $G_\gamma$ is necessarily a positive real scalar multiple of the measure $\lVert\omega\rVert$ associated to a top-degree (left-)invariant differential form $\omega\in\det(\gfr_\gamma)^\vee$. We define the correspondence between Haar measures on $P_\gamma$ and $M_{\overline\gamma}$ by sending $\lVert\omega\rVert$ to $\lVert\overline\omega\rVert$, where $\overline\omega\in \det(\mfr_{\overline\gamma})^\vee$ is the image of $\omega$ under the isomorphism in Lemma~\ref{lem:Laumon4.8.5}. 

    In general, if we only have $\Pbf\supset\Gbf_\gamma$, then we also need to choose a basis for $\det(\nfr_{[\gamma]})$ to set up a correspondence between Haar measures of $G_\gamma$ and $M_{\overline\gamma}$.
\end{rmk}

From now on, let $\gamma \in G=\Urm(V,h)$ be a \emph{$\wp$-primary} element, and endow $V$ with an $E[T]$-module structure induced by $\gamma$. Suppose that we have an $E[T]$-module isomorphism 
\[
V\cong \bigoplus_{i=1}^s \left( E[T]/\wp^{m_i} \right)^{r_i}, \quad \text{as in \eqref{eq:Laumon4.8.4:primary}}.
\] 
In this setting, we have constructed a parabolic subgroup $\Pbf\subset \Gbf$ that contains $\Gbf_\gamma$ (see Definition~\ref{def:assoc-fil} and Lemma~\ref{lem:centraliser-parabolic}).
\begin{cor}\label{cor:openness-levi-preimage} 
    In the above setting, we have
    \[
        \dim(\Gbf_\gamma) - \dim (\Mbf_{\overline\gamma}) = 2 d_\wp\sum_{i=1}^s(m_i-m_{i+1}) \Big(\sum_{i'\leqslant i;\ i' \in I_-} r_{i'} \Big)\cdot \Big(\sum_{i'\leqslant i;\ i' \in I_+} r_{i'} \Big), 
    \]
    using the notation of Lemma~\ref{lem:assoc-fil}(\ref{lem:assoc-fil:gr-dim}).
    
    In particular, $\Ocal_P(\gamma)$ is open in $\Ocal_M(\overline\gamma)\cdot N$ if and only if $m_1,\cdots,m_s$ are either all even or all odd (i.e., either $I_-=\emptyset$ or $I_+=\emptyset$), in which case there is a natural bijection between the Haar measures on $G_\gamma$ and $M_{\overline\gamma}$. This condition is satisfied in particular when $\gamma$ is regular (i.e., when $\dim(\Gbf_\gamma)=\dim_E V$).
\end{cor}
\begin{proof}
The first claim is immediate from the dimension formulae for $\Gbf_\gamma$ \eqref{eq:centraliser-dim} and $\Mbf_{\overline\gamma}$ \eqref{eq:Laumon4.8.4:dim}. Therefore, $\dim(\Gbf_\gamma)=\dim(\Mbf_{\overline{\gamma}})$ if and only if either $I_-$ or $I_+$ is empty, and this condition is satisfied when $\gamma$ is regular (since  $s = |I_-\sqcup I_+| = 1$).
The rest of the claims now follow from Proposition~\ref{prop:openness-levi-preimage}(\ref{prop:openness-levi-preimage:closure}), Lemma~\ref{lem:Laumon4.8.5}, and Remark~\ref{rmk:Laumon4.8.7}. 
\end{proof}

\begin{rmk}\label{rmk:openness-levi-preimage} 
    In the setting of Corollary~\ref{cor:openness-levi-preimage}, if both $I_-$ and $I_+$ are non-empty then we have $\dim(\Gbf_\gamma) - \dim(\Mbf_{\overline\gamma})>0$. In this case, it may \emph{not} be possible to find \emph{another} $F$-rational parabolic subgroup $\Pbf'\subset\Gbf$ containing $\Gbf_\gamma$ such that the image $\overline\gamma\in M'$ under the Levi quotient is \emph{closed} and we have $\dim(\Gbf_\gamma) = \dim(\Mbf_{\overline\gamma})$. For example, let $(V,h)$ be a rank-$3$ non-degenerate hermitian $E$-space. Then $\Gbf=\Ubf(V,h)$ is quasi-split with $F$-semisimple rank~$1$, so any proper parabolic subgroup is a Borel subgroup. Now, choose $\gamma\in G$ so that the corresponding $E[T]$-module $V$ is isomorphic to $E[T]/(T-1)^2\oplus E[T]/(T-1)$. (To construct such $\gamma$, choose a rank-$2$ split hermitian $E$-subspace $(V_0,h_0)$ of $(V,h)$, and let $\gamma$ be the image of a regular unipotent element $\gamma_0\in \Urm(V_0,h_0)$.) Then, we have $\dim(\Gbf_\gamma)= 5$ by \eqref{eq:centraliser-dim}. For $\overline\gamma\in M$ to be closed, $\Pbf$ should be a proper parabolic, which forces $\Pbf$ to be a Borel subgroup and $\overline\gamma = 1$. Thus, $\Mbf_{\overline\gamma}$ is a maximal torus, which has dimension~$3$, so the dimension equality $\dim(\Gbf_\gamma) = \dim(\Mbf_{\overline\gamma})$ cannot be accomplished.
\end{rmk}
\begin{rmk}
    In the setting of Corollary~\ref{cor:openness-levi-preimage}, suppose furthermore that either $I_-$ or $I_+$ is empty so that $\Ocal_P(\gamma)$ is open in $\Ocal_M(\overline\gamma)\cdot N$. Even in this case, $\Ocal_P(\gamma)$ may \emph{not} be dense in $\Ocal_M(\overline\gamma)\cdot N$, unlike the situation for $\GL_E(V)$ in Lemma~\ref{lem:Laumon4.8.4}. We provide an example where $\gamma$ is unipotent, in which case $\Ocal_M(\overline\gamma)\cdot N=N$. The Zariski closure of $\Ocal_P(\gamma)$ is contained in $N$ by the Zariski closedness in $P$, and may be strictly larger than the analytic closure.

    To give a concrete example, let $(V,h)$ be a rank-$4$ split hermitian $E$-space and let $\gamma\in G=\Urm(V,h)$ be a regular unipotent element in some Siegel Levi subgroup. Then $V\cong \left( E[T]/(T-1)^2 \right)^2$, and $\Pbf$ is another Siegel parabolic subgroup associated to $\ker(\gamma-1\mid V)$, so $\Hrm^1(F,\Pbf)$ is trivial. On the other hand, $|\Hrm^1(F,\Gbf_\gamma)|=2$ by Proposition~\ref{prop:unitary-centralisers}(\ref{prop:unitary-centralisers:unip-rad}) and Corollary~\ref{cor:elem-div-self-dual}(\ref{cor:elem-div-self-dual:centraliser-ps}). (In fact, $\Rscr_{u,F}\Gbf_\gamma$ is cohomologically trivial, and $\Gbf_\gamma^{\ps}$ is a rank-$2$ unitary group.) By repeating the proof of Proposition~\ref{prop:stable-conjugacy}, there exist two distinct conjugacy classes $\Ocal_P(\gamma)$ and $\Ocal_P(\gamma')$ of the same dimension, both contained in $N$ as open submanifolds.
\end{rmk}

\begin{subequations}
    \addtocounter{equation}{-1}
Now we construct a closed $F$-submanifold $\Vcal(\gamma)\subset P$ containing $\Ocal_P(\gamma)$ as an open $F$-submanifold. Recall that the embedding $\Ocal_P(\gamma)\hookrightarrow\Ocal_M(\overline\gamma)\cdot N$ induces a map on tangent spaces whose cokernel is $\nfr_{[\gamma]}$, using the notation from Lemma~\ref{lem:Laumon4.8.5}.

We also fix a Levi decomposition $\Pbf = \Mbf\ltimes\Nbf$, which defines the grading on $\gfr$ as in \eqref{eq:wt-Lie}, and the ideal $\nfr_2\subset \nfr$ as in Lemma~\ref{lem:nfr-wt-decomp}.

\begin{prop}\label{prop:tangent-sp-cokernel}
The natural projection $\nfr \twoheadrightarrow \nfr/\nfr_2 \cong \gfr(1)$ induces an isomorphism of $F$-vector spaces:
\begin{equation}\label{eq:tangent-sp-cokernel}
    \begin{tikzcd}
\nfr_{[\gamma]} \arrow[r, "\sim"] &
\coker\left( 1-\Ad(\overline\gamma^{-1}) \mid \gfr(1) \right).
\end{tikzcd}
\end{equation}
\end{prop}
\begin{proof}
    Recall from Corollary~\ref{cor:assoc-fil}(\ref{cor:assoc-fil:N2}) that the image of $\gamma$ on $P/N_2 \cong M \ltimes \gfr(1)$ is equal to $\overline\gamma$. Therefore, the action of $1-\Ad(\gamma^{-1})$ on $\nfr/\nfr_2$ is canonically identified with $1-\Ad(\overline\gamma^{-1})$ on $\gfr(1)$. Thus the composition of natural surjections
    \begin{equation}
        \begin{tikzcd}
            \nfr \arrow[r, two heads] &
            \gfr(1) \arrow[r, two heads] &
            \coker\left( 1-\Ad(\overline\gamma^{-1}) \mid \gfr(1) \right)
        \end{tikzcd}
    \end{equation}
    factors through $\nfr_{[\gamma]}$, which shows the surjectivity of \eqref{eq:tangent-sp-cokernel}.
    
    It remains to show that the source and the target of \eqref{eq:tangent-sp-cokernel} have the same dimension. On the one hand, Lemma~\ref{lem:Laumon4.8.5} yields $\dim\nfr_{[\gamma]} = \dim(\Gbf_\gamma) - \dim (\Mbf_{\overline\gamma})$, which was computed in Corollary~\ref{cor:openness-levi-preimage}. To compute the dimension of the target, recall from Lemma~\ref{lem:nfr-wt-decomp}(\ref{lem:nfr-wt-decomp:odd-wt}) that we have  
    \[\gfr(1)\cong \bigoplus_{k >0} \Hom_E\left(V(k), V(k + 1)\right).\]
    Now, we view each $V(j)$ as an $E_\wp$-vector space via the isomorphism $V(j)\cong \gr^j(V)$; or equivalently, we note that the $E$-subalgebra of $\End_E\big(V(j)\big)$ generated by $\overline\gamma$ is $E_\wp$. This shows that
    \begin{equation}
        \ker\big(1-\Ad(\overline\gamma)^{-1}\mid \gfr(1)\big) \cong \bigoplus_{j\geqslant0}\Hom_{E_\wp}\left( V(j),V(j+1) \right),
    \end{equation} 
    since $E$-linear homomorphisms commuting with $\overline\gamma$ are precisely $E_\wp$-linear homomorphisms. Thus, we have
    \begin{multline}
        \dim_F\left( \coker\big(1-\Ad(\overline\gamma)^{-1}\mid \gfr(1)\big) \right) \\
        \begin{aligned}[b]
            &= \dim_F\left( \ker\big(1-\Ad(\overline\gamma)^{-1}\mid \gfr(1)\big) \right)\\
            &= [E_\wp:F]\cdot \sum_{j\geqslant0} \dim_{E_\wp}\big(V(j)\big)\cdot \dim_{E_\wp}\big(V(j+1)\big) \\
            &= 2 d_\wp\sum_{i=1}^s(m_i-m_{i+1}) \Big(\sum_{i'\leqslant i;\ i' \in I_-} r_{i'} \Big)\cdot \Big(\sum_{i'\leqslant i;\ i' \in I_+} r_{i'} \Big), 
        \end{aligned}
    \end{multline}
    which is equal to $\dim_F\nfr_{[\gamma]}$ by Lemma~\ref{lem:Laumon4.8.5} and Corollary~\ref{cor:openness-levi-preimage}.
\end{proof}
\end{subequations}

\begin{subequations}
    \addtocounter{equation}{-1}
    \begin{constr}
    To construct a closed $F$-submanifold $\Vcal(\gamma)\subset P$ containing $\Ocal_P(\gamma)$, let us first construct the ``approximation'' of $\Vcal(\gamma)$ modulo $N_2$, as follows.
    
    Let $\left( \Ocal_M(\overline\gamma)\cdot N \right)/N_2$ be the quotient with respect to the right translation by $N_2$. Then, we have an isomorphism of $F$-analytic manifolds
    \begin{equation}\label{eq:ambient-vb}
        \begin{tikzcd}
            \left(\Ocal_M(\overline\gamma)\cdot N \right)/N_2
            \arrow[r,"\cong"] &
            M_{\overline\gamma}\backslash \left( M\ltimes\gfr(1) \right)
        \end{tikzcd}
    \end{equation}
    sending $(m^{-1}\overline\gamma m)\cdot n$ to $M_{\overline\gamma}\cdot (m\cdot\xi)$ for any $m\in M$ and $n\in N$, where $\xi\in\gfr(1)$ is the image of $n$ under $N\twoheadrightarrow N/N_2\cong \gfr(1)$. On the source of the map, $M_{\overline\gamma}$ acts on $M$ by left translation and trivially on $\gfr(1)$. We endow $M_{\overline\gamma}\backslash \left( M\ltimes\gfr(1) \right)$ with a right $M$-action by letting $m\in M$ act via the right translation on $M_{\overline\gamma}\backslash M$ and by $\Ad(m^{-1})$ on $\gfr(1)$. With this action, the composition of the natural maps
    \begin{equation}\label{eq:incl-mod-N2}
        \begin{tikzcd}
            \Ocal_P(\gamma) \arrow[r, hook] &
            \Ocal_M(\overline\gamma)\cdot N \arrow[r, two heads] &
            \left(\Ocal_M(\overline\gamma)\cdot N \right)/N_2 \cong M_{\overline\gamma}\backslash \left( M\ltimes\gfr(1) \right)
        \end{tikzcd}
    \end{equation}
    is equivariant for the natural right $M$-actions.

    The natural projection $M_{\overline\gamma}\backslash \left( M\ltimes\gfr(1) \right)\to M_{\overline\gamma}\backslash M$ makes its source a trivial $M$-equivariant vector bundle over $M_{\overline\gamma}\backslash M$ with fibre $\gfr(1)$. 
    We now construct an $M$-equivariant subbundle $\overline\Vcal(\gamma) \subseteq M_{\overline\gamma}\backslash \left( M\ltimes\gfr(1) \right)$ as follows. Firstly, we consider 
    \begin{equation}
        \overline\Vcal_0\coloneqq  \big(1- \Ad(\overline\gamma^{-1}) \big)\big( \gfr(1) \big) = \ker\big(  \gfr(1) = \nfr/\nfr_2 \twoheadrightarrow \nfr_{[\gamma]} \big),
    \end{equation}
    which is an $F$-subspace of $\gfr(1)$ stable under the adjoint action of $M_{\overline\gamma}$.
    Now, define 
    \begin{equation}\label{eq:V-bar-gamma}
         \overline\Vcal(\gamma) \coloneqq M_{\overline\gamma}\backslash\big\{ (m,\xi)\in M\ltimes \gfr(1) \mid \Ad(m)(\xi)\in \overline\Vcal_0 \big\} \twoheadrightarrow M_{\overline\gamma}\backslash M,
    \end{equation}
    which is an $M$-equivariant subbundle of $M_{\overline\gamma}\backslash \left( M\ltimes\gfr(1) \right)$ such that the fibre at $M_{\overline\gamma}m \in M_{\overline\gamma}\backslash M$ is $\Ad(m^{-1})(\overline\Vcal_0)$. 
    \end{constr}
\end{subequations}

\begin{lem}\label{lem:V-bar-gamma}
The subbundle $\overline\Vcal(\gamma)$ is the image of $\Ocal_P(\gamma)$ under the map \eqref{eq:incl-mod-N2}.
\end{lem}
\begin{proof}
    By Corollary~\ref{cor:assoc-fil}(\ref{cor:assoc-fil:N2}), the image of $\gamma$ and $\overline\gamma$ in $P/N_2$ coincides, so let us first describe $\Ocal_{P/N_2}(\overline\gamma)$.

    Recall that the $\Gbf(E)$-equivariant isomorphism
    \[
    \begin{tikzcd}[column sep=small]
        \Nbf(E)/\Nbf_2(E)\arrow[rr,"\cong"] & & \gfr(1)_E& \bar{n} ;\arrow[r, maps to] & \xi\coloneqq \bar{n}-1 &\text{(see \eqref{eq:nfr-wt-decomp:exp-gfr-1})}
    \end{tikzcd}\]
    restricts to $N/N_2\riso \gfr(1)$; 
    see Lemma~\ref{lem:nfr-wt-decomp}(\ref{lem:nfr-wt-decomp:exp}) and its proof. Thus, for any $\xi\in \gfr(1)$, we have
    \[
    (1-\xi)\overline\gamma(1+\xi) = \overline\gamma\cdot \left( 1+\xi -\Ad(\overline\gamma^{-1})(\xi) -\Ad(\overline\gamma^{-1})(\xi)\cdot\xi \right) \in \Pbf(E).
    \]
    Now, since $\Ad(\overline\gamma^{-1})(\xi)\cdot\xi\in\nfr_{2,E}$, the image of $(1+\xi)^{-1}\overline\gamma(1+\xi)$ in $\Pbf(E)/\Nbf_2(E)$ is 
    \[
    \overline\gamma\cdot \left( 1+\xi -\Ad(\overline\gamma^{-1})(\xi)  \right) \in P/N_2.
    \]
    Thus, the isomorphism $P/N_2\riso M\ltimes\gfr(1)$ restricts to
    \[
    \begin{tikzcd}
        \{\bar{n}^{-1}\overline\gamma\bar{n}\mid \bar{n}\in N/N_2\} \arrow[r,"\cong"] & \overline\gamma\cdot \overline\Vcal_0.
    \end{tikzcd}
    \]
    By conjugating the above expression by $M$ and comparing it with the definition of $\overline\Vcal(\gamma)$ in \eqref{eq:V-bar-gamma}, the lemma follows.
\end{proof}

\begin{subequations}
    \addtocounter{equation}{-1}
\begin{cor}\label{cor:V-gamma}
    Let $\Vcal(\gamma)$ denote the full preimage of $\overline\Vcal(\gamma)$ under the composition
    \begin{equation}
        \begin{tikzcd}
        \Ocal_M(\overline\gamma)\cdot N \arrow[r, two heads] & \left(\Ocal_M(\overline\gamma)\cdot N \right)/N_2 \arrow[r,"\cong"] & M_{\overline\gamma}\backslash \left( M\ltimes\gfr(1) \right).
    \end{tikzcd}
    \end{equation}
    Then, $\Vcal(\gamma)\subset \Ocal_M(\overline\gamma)\cdot N$ is a closed $F$-submanifold stable under the conjugation by $P$, and it contains $\Ocal_P(\gamma)$ as an open $F$-submanifold.
\end{cor}
\begin{proof}
    The closedness of $\Vcal(\gamma)$ in $\Ocal_M(\overline\gamma)\cdot N$ is immediate since it is the full preimage of a subbundle $\overline\Vcal(\gamma)\subset \left(  \Ocal_M(\overline\gamma)\cdot N \right)/N_2$, and it contains $\Ocal_P(\gamma)$ by Lemma~\ref{lem:V-bar-gamma}. The stability under the conjugation by $P$ can be checked modulo a normal subgroup $N_2$, which is immediate from Lemma~\ref{lem:V-bar-gamma}.  
    
    It remains to show the openness of $\Ocal_P(\gamma)$ in $\Vcal(\gamma)$. Set 
    \begin{equation}\label{eq:V0}
        \Vcal_0\coloneqq \ker(\nfr\twoheadrightarrow\nfr_{[\gamma]}),
    \end{equation}
     which is the full preimage of $\overline\Vcal_0\subset \gfr(1)\cong \nfr/\nfr_2$. Then, the analytic embeddings $\Ocal_P(\gamma) \hookrightarrow \Vcal(\gamma) \hookrightarrow\Ocal_M(\overline\gamma)\cdot N$ induce the following maps on the tangent spaces:
    \begin{equation}\label{eq:embedding-tang-sp-Vcal}
        \begin{tikzcd}
        \gfr_\gamma\backslash\pfr\arrow[r, hook]  & 
        (\mfr+\Vcal_0) \times_{\mfr} (\mfr_{\overline\gamma}\backslash\mfr) \arrow[r,hook]  & 
        \pfr \times_{\mfr} (\mfr_{\overline\gamma}\backslash\mfr).
        \end{tikzcd}
    \end{equation}
    Now, observe that the first two terms have the same codimension
     $\dim_F(\nfr_{[\gamma]})$ in the last term by Proposition~\ref{prop:openness-levi-preimage} and Lemma~\ref{lem:Laumon4.8.5}, so we have 
    \begin{equation}\label{eq:isom-tang-sp-Vcal}
        \begin{tikzcd}
        \gfr_\gamma\backslash\pfr\arrow[r, "\cong"]  & 
        (\mfr+\Vcal_0) \times_{\mfr} (\mfr_{\overline\gamma}\backslash\mfr) .
        \end{tikzcd}
    \end{equation}
    This shows that the embedding $\Ocal_P(\gamma)\hookrightarrow\Vcal(\gamma)$ is a local isomorphism, so it is an open immersion.
\end{proof}
\end{subequations}

\begin{rmk}\label{rmk:V-gamma-regular}
    By Corollary~\ref{cor:openness-levi-preimage}, we have $\nfr_{[\gamma]}= 0$ if and only if $N_2 = N$. In this case, we clearly have $\Vcal(\gamma) = \Ocal_M(\overline\gamma)\cdot N$.
\end{rmk}

\begin{rmk}\label{rmk:unip-orbit}
If $\overline\gamma$ is in the centre of $M$, which is the case when $\gamma$ is unipotent, then $\overline\Vcal(\gamma)$ is trivial and $\Vcal(\gamma) = N_2$. The analogous result for the $\Ad(P)$-orbit of a nilpotent Lie algebra element was obtained by R.~Rao \cite{RangaRao:OrbInt} in characteristic~$0$ and G.~McNinch \cite[Proposition~46]{McNinch:NilpOrbits} in any characteristic, which inspired our construction of $\Vcal(\gamma)$. However, it should be warned that in our setting $\Vcal(\gamma)$ cannot be written as a product of $\Ocal_M(\overline\gamma)$ with a subgroup of $N$ in general.
\end{rmk}

\begin{subequations}
    \addtocounter{equation}{-1}
\begin{constr}\label{constr:measure}
    Next, we define a measure on $\Vcal(\gamma)$, depending on the choice of Haar measures $dm$ on $M$, $dn$ on $N$, $dm_{\overline\gamma}$ on $M_{\overline\gamma}$, and $d \xi_{[\gamma]}$ on $\nfr_{[\gamma]}$, which we fix.

    By identifying $\gfr(1)\cong N/N_2$, we can view $\nfr_{[\gamma]}$ as a quotient group of $N$. We set
    \begin{equation}
        N'\coloneqq\ker\left( N\twoheadrightarrow \nfr_{[\gamma]} \right).
    \end{equation}
    The choice of Haar measures $dn$ on $N$ and $d\xi_{[\gamma]}$ on $\nfr_{[\gamma]}$ then determines a Haar measure $dn'$ on $N'$. Note also that the Lie algebra of $N'$ is $\Vcal_0$ in \eqref{eq:V0}, and $\Vcal(\gamma)$ is a principal homogeneous space over $\Ocal_M(\overline\gamma)$ under the action of $N'$.

    Define locally constant characters $\delta_{[\gamma]},\delta_{\Vcal_0}\colon P \to \RR_{>0}^\times$ as follows:
    \begin{align}
        \label{eq:delta-gamma}   
        \delta_{[\gamma]}(g)&\coloneqq \lVert \det\nolimits_F\left( \Ad(g^{-1})\mid \nfr_{[\gamma]} \right)\rVert \quad\text{and} \\
        \label{eq:delta-Vcal-0}     
        \delta_{\Vcal_0}(g)&\coloneqq \lVert \det\nolimits_F\left( \Ad(g^{-1})\mid \Vcal_0\right)\rVert, \quad
        \forall g\in P. 
    \end{align}
    Since the adjoint action of $N$ on $\nfr$ is unipotent, both characters above factor through the Levi quotient $M$.

    Choose an analytic open chart $\{U_\alpha\}$ of $M_{\overline\gamma}\backslash M$ with the following properties.
    \begin{enumerate}
        \item For each $U_\alpha$, there exists a local section $s_\alpha\colon U_\alpha\to M$ of $M\twoheadrightarrow M_{\overline\gamma}\backslash M$. (See \cite[Part~II, Ch~III, \S10]{Serre:LieThy} for the existence.)
        \item $\Vcal(\gamma)\to \Ocal_M(\overline\gamma)\cong M_{\overline\gamma}\backslash M$ admits a local section over each $U_\alpha$.
        \item The characters $\delta_{[\gamma]}|_M$ and $\delta_{\Vcal_0}|_M$ are constant on $s_\alpha(U_\alpha)$ for each $U_\alpha$.
    \end{enumerate}
    Let $D_\alpha$ denote the common value $\delta_{\Vcal_0}(m)$ for any $m\in s_\alpha(U_\alpha)\subset M$. 
    
    Now, we define a measure $d\mu_\alpha$ on $U_\alpha\times N'$ as 
    \begin{equation}\label{eq:d-mu-alpha}
       \left. \frac{dm}{dm_{\overline\gamma}}\right|_{U_\alpha} \cdot \left( D_\alpha\cdot dn' \right).
    \end{equation}
    Using the following open embedding $U_\alpha\times N'\to \Vcal(\gamma)$ defined as
    \begin{equation}\label{eq:local-chart-V-gamma}
        (m, n') \mapsto (m^{-1}\overline\gamma m)\cdot (m^{-1}n' m), \qquad \text{where }m\in s_\alpha(U_\alpha),\ n'\in N',
    \end{equation}
    one can glue the locally defined measures $d\mu_\alpha$ to a measure $d\mu$ on $\Vcal(\gamma)$.
 %
\end{constr}
\end{subequations}

\begin{lem}\label{lem:dmu-multiplier}
For any $g\in P$, we have $d\mu(g^{-1}xg)=\delta_{\Vcal_0}(g)d\mu(x)$.
\end{lem}
\begin{proof}
    If $g = m\in M$, the lemma follows immediately from the construction of $d\mu$. For $g=n\in N$, we claim that $d\mu(x) = d\mu(n^{-1}xn)$. Indeed, by Lemma~\ref{lem:V-bar-gamma}, for any $m\in M$ and $n\in N$ we have
    \[
    n^{-1}(m^{-1}\overline\gamma m)n = (m^{-1}\overline\gamma m)\cdot u
    \]
    for some $u\in m^{-1}N' m$. Thus, based on the local description of $d\mu$ in \eqref{eq:d-mu-alpha}, the effect of conjugation by $n^{-1}$ on $d\mu_\alpha$ is simply the left translation of $D_\alpha\cdot dn'$ by the element $mum^{-1}\in N'$. Since $dn'$ is a Haar measure on $N'$, this translation preserves the measure, which proves the desired invariance. Finally, we observe that $\delta_{\Vcal_0}(n) = \lVert\det_F(\Ad(n^{-1})\mid \Vcal_0)\rVert= 1$ because $n$ acts unipotently on the Lie algebra, completing the proof.
\end{proof}

Note that the fixed Haar measures $dm$ on $M$ and $dn$ on $N$ uniquely determines a left Haar measure $dm\cdot dn$ on $P$. Also, the fixed Haar measures $dm_{\overline\gamma}$ on $M_{\overline\gamma}$ and $d\xi_{[\gamma]}$ on $\nfr_{[\gamma]}$ uniquely determines a bi-invariant Haar measure $dg_\gamma$ on $G_\gamma$ by Lemma~\ref{lem:Laumon4.8.5} (as explained in Remark~\ref{rmk:Laumon4.8.7}). We conclude the section by comparing the quotient measure $\frac{dm\cdot d_n}{dg_{\gamma}}$ on $G_\gamma\backslash P$ with the restriction of $d\mu$.

\begin{prop}\label{prop:measures} \begin{enumerate}
    \item\label{prop:measures:modulus} The restriction $\delta_{[\gamma]}|_{G_\gamma}$ is trivial.
    \item\label{prop:measures:uniqueness} There exists a constant $C>0$ such that for any compact open subset $U\subset G_\gamma\backslash P$, we have
    \[
    C\cdot\int_{G_\gamma\backslash P} \triv_U(mn)\frac{dm\cdot dn}{dg_{\gamma}} = \int_{G_\gamma\backslash P}\delta_{[\gamma]}(g)\triv_U(g)\cdot c_\gamma^\ast (d\mu),
    \]
    where $\triv_U$ denotes the character function of $U$, and $c_\gamma^\ast (d\mu)$ is the pullback of $d\mu$ by $c_\gamma\colon g\mapsto g^{-1}\gamma g$.
\end{enumerate}
\end{prop}
\begin{proof}
    To prove (\ref{prop:measures:modulus}), we consider the determinant of $\Ad(g^{-1})$ for $g\in G_\gamma$ on the $5$-term exact sequence in Lemma~\ref{lem:Laumon4.8.5}. Firstly, by unimodularity we have 
    \[\lVert\det\nolimits_F\big( \Ad(g^{-1})\mid \gfr_\gamma \big)\rVert = \lVert\det\nolimits_F\big( \Ad(\overline g^{-1})\mid \mfr_{\overline\gamma} \big)\rVert = 1,\]
    where $\overline g\in M_{\overline\gamma}$ is the image of $g$. Also $\det_F(\Ad(g^{-1}))$ is identical on the kernel and the cokernel of $1-\Ad(\gamma^{-1})\colon \nfr\to\nfr$. Since the alternating product of determinants of $\Ad(g^{-1})$ on an exact sequence is trivial, we deduce that $\delta_{[\gamma]}(g) = 1$.
    
    Now, let us show (\ref{prop:measures:uniqueness}). Let $\delta_{[\gamma]}\cdot c_\gamma^\ast(d\mu)$ denote the measure on $G_\gamma\backslash P$ defined by the right side of the formula in (\ref{prop:measures:uniqueness}), which is well defined by ~(\ref{prop:measures:modulus}). We claim that the the right translation by $g\in P$ scales the measure $\delta_{[\gamma]}\cdot c_\gamma^\ast(d\mu)$ by $\delta_P(g)$. In fact, since we have $c_\gamma(xg) = g^{-1}c_\gamma(x) g$, the right translation by any $g\in P$ scales the measure $\delta_{[\gamma]}\cdot c_\gamma^\ast(d\mu)$ by the factor $\delta_{[\gamma]}(g)\cdot\delta_{\Vcal_0}(g)$ by Lemma~\ref{lem:dmu-multiplier}. Now, the multiplier can be rewritten as
    \[
    \delta_{[\gamma]}(g)\cdot\delta_{\Vcal_0}(g) = \lVert\det\nolimits_F\left( \Ad(g^{-1})\mid \nfr \right) \rVert = \lVert\det\nolimits_F\left( \Ad(g^{-1})\mid \pfr \right) \rVert = \delta_P(g).
    \]
    To see the second equality, it suffices to show $\lVert\det\nolimits_F\left( \Ad(g^{-1})\mid \mfr \right) \rVert=1$; in fact, this holds for $g\in M$ since $M$ is unimodular, and for $g\in N$ as $\Ad(g^{-1})$ acts unipotently on $\mfr$.

  Therefore, the right translation by $g\in P$ scales both $\delta_{[\gamma]}\cdot c_\gamma^\ast(d\mu)$ and $\frac{dm\cdot dn}{dg_{\gamma}}$ by $\delta_P(g)$.  Thus, by the uniqueness result \cite[Chap.~VII, \S2, No~6, Th~3]{Bourbaki:Integration7-8}, there exists a constant $C>0$ such that $C\cdot \frac{dm\cdot dn}{dg_\gamma} = \delta_{[\gamma]}\cdot c_\gamma^\ast(d\mu)$. This proves (\ref{prop:measures:uniqueness}).
\end{proof}

\section{Proof of the main theorem}\label{sec:proof}
In this section, we show the absolute convergence of the orbital integral $O_\gamma^G(f)$ when $\gamma$ is primary (Theorem~\ref{th:main-primary}), and thus prove Theorem~\ref{th:main}.

To utilise the closed submanifold $\Vcal(\gamma)\subset P$ constructed in Corollary~\ref{cor:V-gamma}, we need the following notion.
\begin{defn}\label{def:good-position}
    Let $\Gbf$ be a connected reductive group over $F$, with $G = \Gbf(F)$. Let $P$ and $K$ be a parabolic subgroup and a compact open subgroup of $G$, respectively. We say that $P$ and $K$ are in \emph{good position} if we have $G = P K$ and $P\cap K = (M\cap K)\ltimes (N\cap K)$ for some Levi decomposition $P = M\ltimes N$.
\end{defn}

\begin{exa}
    Definition~\ref{def:good-position} was originally introduced in \cite[p.~91]{Laumon:Book1} for $G=\GL_E(V)$. In this case, the choice of a parabolic subgroup $P$ with Levi decomposition $P=M\ltimes N$ corresponds to choosing a filtration $\Fil^\bullet V$ and a splitting $\gr^\bullet V$. We can construct a compact open subgroup $K\subset G$ in good position with $P$ as follows: choose an $\Oscr$-lattice $\Lambda$ in $V$ such that $(\Lambda\cap\gr^j V )_j$ defines a grading on $\Lambda$. (For example, let $\Lambda$ be the $\Oscr$-lattice generated by an $F$-basis for $V$ adapted to the grading $\gr^\bullet V$.) Then, the stabiliser $K$ of $\Lambda$ is in good position with $P$.
\end{exa}

In general, we can find a compact open subgroup $K\subset G$ in good position with a given parabolic $P$ via Bruhat--Tits theory.

\begin{prop}\label{prop:special-parahoric-good-position}
    Let $\Gbf$ be a connected reductive group over $F$. Then for any $F$-rational parabolic subgroup $\Pbf\subset \Gbf$ with Levi decomposition $\Pbf = \Mbf\ltimes\Nbf$, there exists a \emph{special} parahoric subgroup $K\subset G$ such that $G = P K$ and $P\cap K = (M\cap K)\ltimes (N\cap K)$. In particular, such $P$ and $K$ are in good position.
\end{prop}

\begin{proof}
    Let $\Sbf$ be a maximal \emph{split} $F$-torus of $\Mbf$, which is necessarily a maximal split $F$-torus of $\Gbf$. We choose a special vertex $x\in\Acal(\Sbf)$, where $\Acal(\Sbf)$ is the apartment corresponding to $\Sbf$ in the Bruhat--Tits building of $\Gbf$. This choice defines a \emph{special parahoric group scheme} $\Gscr\coloneqq\Gscr^0_x$ over $\Oscr$, and the corresponding special parahoric subgroup $K = \Gscr(\Oscr)$. By \cite[Theorem~5.3.4]{KalethaPrasad:BruhatTits}, we have the Iwasawa decomposition $G = P K$, so it remains to verify that $P\cap K = (M\cap K)\ltimes (N\cap K)$.

    Let $F^{\ur}\subset \scl F$ be the maximal unramified subextension, with valuation ring $\Oscr^{\ur}$. By \cite[Proposition~9.3.4]{KalethaPrasad:BruhatTits}, there exists an $F$-torus $\Tbf\subset\Mbf$ containing $\Sbf$ such that $\Tbf_{F^{\ur}}$ is a maximal $F^{\ur}$-split torus of $\Mbf_{F^{\ur}}$ (and hence of $\Gbf_{F^{\ur}}$). The closure $\Tscr_{\Oscr^{\ur}}$ of $\Tbf_{F^{\ur}}$ in $\Gscr_{\Oscr^{\ur}}$ is an $\Oscr^{\ur}$-split torus, which follows from the unramified descent of parahoric groups and the easy direction of \cite[Proposition~9.3.5]{KalethaPrasad:BruhatTits}. Furthermore, $\Tscr_{\Oscr^{\ur}}\subset \Gscr_{\Oscr^{\ur}}$ descends to a closed $\Oscr$-torus $\Tscr\subset\Gscr$, and the schematic closure $\Sscr\subset\Tscr$ of $\Sbf$ is a closed $\Oscr$-split torus. (This deduction relies on unramified descent and the $\Gal(F^{\ur}/F)$-stability of $\Tscr_{\Oscr^{\ur}}\subset \Gscr_{\Oscr^{\ur}}$.)

    Given $\Pbf = \Mbf\ltimes\Nbf$, we can find an $F$-rational cocharacter $\lambda\colon \GG_m \to \Sbf$ such that $\Mbf = \Zbf_\Gbf(\lambda)$ is the centraliser of $\lambda$, and $\Pbf= \Pbf_\Gbf(\lambda)$ is the subgroup of elements $g$ for which the limit $\lim\limits_{t\to 0}\lambda(t)g\lambda(t)^{-1}$ exists. Consequently, $\Nbf = \Nbf_\Gbf(\lambda)$ is the subgroup of elements $g$ such that $\lim\limits_{t\to 0}\lambda(t)g\lambda(t)^{-1} = 1$. (See \cite[Lemma~2.1.4]{ConradGabberPrasad:PRedGp2ed} for the precise definition of the limit.) 
    
    Because $\Sscr$ is a split $\Oscr$-torus, $\lambda$ extends uniquely to a homomorphism $\GG_{m,\Oscr}\to\Sscr$ over $\Oscr$, which we also denote by $\lambda$. We can then analogously define the  closed $\Oscr$-subgroups $\Pscr\coloneqq \Pbf_{\Gscr}(\lambda)$, $\Mscr\coloneqq \Zbf_\Gscr(\lambda)$, and $\Nscr\coloneqq \Nbf_\Gscr(\lambda)$ of $\Gscr$, which are smooth $\Oscr$-models of $\Pbf$, $\Mbf$, and $\Nbf$, respectively (see \cite[Lemmas~2.1.4, 2.1.5]{ConradGabberPrasad:PRedGp2ed}). Furthermore, by \cite[Proposition~2.1.8(3)]{ConradGabberPrasad:PRedGp2ed}, the multiplication map induces a scheme-theoretic isomorphism:
    \[\begin{tikzcd}
      \Mscr\ltimes\Nscr \arrow[r,"\cong"] & \Pscr .
    \end{tikzcd}\]

    Finally, observe that taking $\Oscr$-points yields
    \[
    \Pscr(\Oscr) = \Pbf(F)\cap \Gscr(\Oscr) = P\cap K.
    \]
    Similarly, we have $\Mscr(\Oscr) = M\cap K$ and $\Nscr(\Oscr) = N\cap K$. Evaluating the scheme-theoretic Levi decomposition on $\Oscr$-points immediately gives the desired point-set decomposition $P\cap K = (M\cap K)\ltimes(N\cap K)$.
\end{proof}

\begin{rmk}
If $\Gbf = \Ubf(V,h)$ is a unitary group where $E/F$ is not wildly ramified (that is, either the characteristic of $F$ is not $2$ or $E/F$ is unramified), then given a parabolic subgroup $\Pbf \subset \Gbf$ with a fixed Levi decomposition, one can explicitly construct a special parahoric subgroup $K\subset G$ in good position with $P$. This construction relies on the concrete presentation of special parahoric subgroups of $G$ provided in \cite[\S3]{GanHankeYu:MassFormula}.
\end{rmk}

Now we are ready to prove the main result.
\begin{thm}\label{th:main-primary}
    Let $\Gbf = \Ubf(V,h)$ be a unitary group over a non-archimedean local field $F$. Let $\gamma\in G$ be a primary element. Then for any $f\in C^\infty_c(G)$, the orbital integral $\Ocal^G_\gamma(f)$ absolutely converges. In fact, it reduces to a finite sum.
\end{thm}
\begin{proof}
    To each primary element $\gamma\in G$, we have constructed a parabolic subgroup $\Pbf\subset\Gbf$ in Definition~\ref{def:assoc-fil}. Choose a Levi decomposition $\Pbf = \Mbf\ltimes\Nbf$ as in Definition~\ref{def:assoc-fil}, and choose a special parahoric subgroup $K\subset G$ in good position with $P$ using Proposition~\ref{prop:special-parahoric-good-position}. 

    Choose left Haar measures $dm$ on $M$ and $dn$ on $N$ so that $M\cap K$ and $N\cap K$ have volume~$1$, respectively. We also fix Haar measures $dm_{\overline\gamma}$ on $M_{\overline\gamma}$ and $d\xi_{[\gamma]}$ on $\nfr_{[\gamma]}$. By Lemma~\ref{lem:Laumon4.8.5} and Remark~\ref{rmk:Laumon4.8.7}, $dm_{\overline\gamma}$ and $d\xi_{[\gamma]}$ uniquely determine a Haar measure $dg_\gamma$ on $G_\gamma$. Furthermore, by Construction~\ref{constr:measure} we also obtain a unique measure $d\mu$ on $\Vcal(\gamma)$, where $\Vcal(\gamma)$ is defined in Corollary~\ref{cor:V-gamma}. Lastly, choose a Haar measure $dg$ on $G$ so that $K$ has volume $1$. We may view $dg$ as a Haar measure on $K$.
    
    Now, for any $f\in C^\infty_c(G)$ we get
    \begin{align}
         O^G_\gamma(f) = \int_{G_\gamma\backslash G}f(g^{-1}\gamma g) \frac{dg}{dg_\gamma} &= \int_{G_\gamma\backslash P}\int_K f(k^{-1}h^{-1}\gamma h k) dk\cdot\frac{dh}{dg_\gamma}\\
         & = \int_{G_\gamma\backslash P} f_K(h^{-1}\gamma h)\frac{dh}{dg_\gamma},\notag
    \end{align}
    where $f_K(h)\coloneqq \int_K f(k^{-1}hk)dk$. Note that $f_K\in C^\infty_c(G)$. (If $f$ is constant on $gK'$ for some open normal subgroup $K'\subset K$, then $\bigcap_{k\in K}(kgk^{-1})K'$ is essentially a finite intersection where $f_K$ is constant. If $f$ is supported on a finite union of $Kg_i K$'s, then so is $f_K$.) In particular, we have $f_K|_P\in C^\infty_c(P)$.   

    Recall that $\Vcal(\gamma)$ is a closed $F$-submanifold of $P$, containing $\Ocal_P(\gamma)$ as an open $F$-submanifold, by Corollary~\ref{cor:V-gamma}. Now write the support of $f_K|_{\Vcal(\gamma)}$ as a disjoint union $\bigsqcup_\beta S_\beta$ for some closed compact subset $S_\beta\subset \Vcal(\gamma)$ such that 
    \[f_K|_{\Vcal(\gamma)} = \sum_\beta a_\beta\triv_{S_\beta}\]  
    where $a_\beta\in\CC$, and it satisfies the following property: choosing an open chart $\{U_\alpha\}$ of $M_{\overline\gamma}\backslash M$ as in Construction~\ref{constr:measure}, each $S_\beta$ is contained in the image of the open embedding $U_\alpha\times N'\to \Vcal(\gamma)$ defined in \eqref{eq:local-chart-V-gamma} for some $U_\alpha$. By construction, $\delta_{[\gamma]}$ \eqref{eq:delta-gamma} is constant on each $S_\beta$, so we let $C_\beta$ denote the common value of $\delta_{[\gamma]}(g)$ for $g\in\ S_\beta$.

    Now, we set
    \[S_\beta^\circ\coloneqq S_\beta \cap \Ocal_P(\gamma).\]
    Then, we have 
    \[
    \vol\Big(S_\beta; d\mu\Big) \leqslant \vol\Big(S_\beta; d\mu\Big) <\infty;
    \]
    in fact, the boundedness of the volume of $S_\beta$ is by compactness, and the inequality is due to the inclusion $S_\beta^\circ\subseteq S_\beta$.
    Thus, by Proposition~\ref{prop:measures} we have
    \[
        \vol \Big( S^\circ_\beta; \frac{dm\cdot dn}{d g_\gamma} \Big) = C^{-1}\cdot C_\beta\cdot\vol\Big(S_\beta^\circ; d\mu\Big) <\infty, 
    \]
    where $C$ is the positive constant in Proposition~\ref{prop:measures}(\ref{prop:measures:uniqueness}). 

    In conclusion, we obtain
    \[
    O_\gamma^G(f) = \sum_\beta a_\beta\cdot  \vol\Big(S_\beta^\circ; \frac{dm\cdot dn}{d g_\gamma}\Big),
    \]
    which is a finite sum of complex numbers.
\end{proof}

We are ready to show the absolute convergence of $O^G_\gamma(f)$ for any $\gamma\in G$ and $f\in C^\infty_c(G)$, and conclude the proof of Theorem~\ref{th:main}.
\begin{proof}[Proof of {Theorem~\ref{th:main}}]
    By Corollary~\ref{cor:Gs}, to show the absolute convergence of $O_\gamma^G(f)$ it suffices to show the absolute convergence when $\gamma$ is primary, which is proved in Theorem~\ref{th:main-primary}.
    \end{proof}

\subsection*{Acknowledgement}
This work was supported by the National Research Foundation of Korea(NRF) grant funded by the Korea government(MSIT). (RS-2026-25469397). The authors acknowledge the use of Google's Gemini for English language editing and literature search during the preparation of this manuscript.

  \printbibliography{}
\end{document}